\documentclass[aap,MSNbibl,seceqn,dvips]{arximspdf}
\usepackage{mathrsfs}
\usepackage{tikz}
\usepackage{graphicx}

\doi{10.1214/13-AAP996} 
\volume{25}
\issue{1}
\pubyear{2015}
\firstpage{287}
\lastpage{323}

\makeatletter

\newcommand{\rrvert}{\vert}
\newcommand{\llvert}{\vert}

\newcommand{\E}[0]{\mathbb{E}}
\newcommand{\vect}[1]{{ #1}}
\newcommand{\deq}[0]{\stackrel{d}{=}}
\newcommand{\Cov}{\operatorname{Cov}}
\newcommand{\one}{\mathbf{1}}
\newcommand{\prob}{\mathbb{P}}
\newcommand{\N}{\mathbb{N}}
\newcommand{\R}{\mathbb{R}}
\newcommand{\D}{\mathrm{Plane\ Active}}
\newcommand{\threshold}{\theta}
\newcommand{\vertical}{\mathrm{vert}}
\newcommand{\horizontal}{\mathrm{horiz}}
\newcommand{\plane}{P}
\newcommand{\noodle}{r}
\newcommand{\abovethresh}{\mathrm{Above\ Threshold}}

\newcommand{\basic}{\operatorname{\texttt{Basic}}}
\newcommand{\lin}{\operatorname{\texttt{Line}}}
\newcommand{\oline}{\operatorname{\texttt{Line}\emptyset}}
\newcommand{\eline}{\operatorname{\texttt{EnhancedLine}}}
\newcommand{\ebasic}{\operatorname{\texttt{EnhancedBasic}}}
\newcommand{\neline}{\operatorname{\texttt{NonEnhancedLine}}}
\newcommand{\good}{\mathrm{good}}
\newcommand{\event}{G}

\newcommand{\cF}{\mathcal F}

\newcommand{\Gone}[1]{\operatorname{Cross\ Lines}_m(#1)}
\newcommand{\Gtwo}[1]{\operatorname{Cross\ Lines}_r(#1)}
\newcommand{\Gonez}{\operatorname{Cross\ Lines}_m}
\newcommand{\Gtwoz}{\operatorname{Cross\ Lines}_r}

\newcommand{\lineset}{\mathscr{L}}
\newcommand{\Ln}{\ell}

\newtheorem{theorem}{Theorem}[section]

\newtheorem{lemma}[theorem]{Lemma}
\newproclaim{question}[theorem]{Question}
\newproclaim{remark}[theorem]{Remark}

\newcommand{\comment}[1]{}

\newcommand{\binom}[2]{\pmatrix{#1\cr #2}}
\newcommand{\Bin}{\operatorname{Bin}}

\renewcommand{\epsilon}{\varepsilon}
\newcommand{\Initial}{\operatorname{Initial}}
\newcommand{\Binomial}{\operatorname{Binomial}}
\def\sfrac#1#2{#1/#2}
\def\vfrac#1#2{(#1)/#2}

\def\vafrac#1#2{(#1)/(#2)}

\makeatother

\begin{document}
\begin{frontmatter}

\title{Bootstrap percolation on the Hamming torus}
\runtitle{Bootstrap percolation on the Hamming torus}

\begin{aug}
\author[A]{\fnms{Janko} \snm{Gravner}\corref{}\thanksref{T1}\ead[label=e1]{gravner@math.ucdavis.edu}},
\author[B]{\fnms{Christopher} \snm{Hoffman}\thanksref{T2}\ead[label=e2]{hoffman@math.washington.edu}},
\author[B]{\fnms{James}~\snm{Pfeiffer}\thanksref{T3}\ead[label=e3]{jpfeiff@math.washington.edu}}
\and
\author[C]{\fnms{David} \snm{Sivakoff}\thanksref{T4}\ead[label=e4]{djsivy@math.duke.edu}}
\runauthor{Gravner, Hoffman, Pfeiffer and Sivakoff}
\affiliation{University of California, Davis, University of Washington,\\
University of Washington and~Duke~University}
\address[A]{J. Gravner\\
Department of Mathematics\\
University of California\\
Davis, California 95616\\
USA\\
\printead{e1}} 

\address[B]{C. Hoffman\\
J. Pfeiffer\\
Department of Mathematics\\
University of Washington\\
Seattle, Washington 98195\\
USA\\
\printead{e2}\\
\hphantom{\textsc{E-mail}:} \printead*{e3}}

\address[C]{D. Sivakoff\\
Mathematics Department\\
Duke University\\
Durham, North Carolina 27708\\
USA\\
\printead{e4}}
\end{aug}
\thankstext{T1}{Supported in part by NSF Grant DMS-02-04376 and the
Republic of Slovenia's Ministry of Science program P1-285.}
\thankstext{T2}{Supported in part by NSF Grant DMS-08-06024 and by
an AMS Centennial Fellowship.}
\thankstext{T3}{Supported in part by NSF Grant DMS-11-15293.}
\thankstext{T4}{Supported in part by NSF Grant DMS-10-57675 and by SAMSI.}

\received{\smonth{9} \syear{2012}}
\revised{\smonth{11} \syear{2013}}

%
\begin{abstract}
The Hamming torus of dimension $d$ is the graph with vertices $\{
1,\dots, n\}^d$ and an edge between any two vertices that differ in a
single coordinate. Bootstrap percolation with threshold $\theta$
starts with a random set of open vertices, to which every vertex
belongs independently with probability $p$, and at each time step the
open set grows by adjoining every vertex with at least $\theta$ open
neighbors. We assume that $n$ is large and that $p$ scales as
$n^{-\alpha}$ for some $\alpha>1$, and study the probability that an
$i$-dimensional subgraph ever becomes open. For large $\theta$, we
prove that the critical exponent $\alpha$ is about $1+d/\theta$ for
$i=1$, and about $1+2/\theta+\Theta(\theta^{-3/2})$ for $i\ge2$.
Our small $\theta$ results are mostly limited to $d=3$, where we
identify the critical $\alpha$ in many cases and, when $\theta=3$,
compute exactly the critical probability that the entire graph is
eventually open.
\end{abstract}


\end{frontmatter}

\section{Introduction}

Bootstrap percolation is a simple growth model, introduced to
understand nucleation and metastability in physical processes such as
crack formations, clustering and alignment of magnetic
spins. It was introduced in 1979 by Chalupa, Leath and Reich \cite
{bethe}. For more applications and background, see surveys by Adler and
Lev \cite{brazil} and Holroyd \cite{holroyd-survey}.

Given a graph $G=(V,E)$, \textit{bootstrap percolation} with \textit
{threshold} $\threshold$
is the following discrete-time growth process: given an initial
configuration $\omega\in
\{0,1\}^V$, an increasing sequence of configurations
$\omega=\omega_0,\omega_1,\ldots$ is defined by\vspace*{-1pt}
\[
\omega_{j+1}(v)= %
\cases{ 1,&\quad \mbox{if} $
\omega_j(v)=1$ \mbox{or} $\displaystyle \sum_{w \sim v}
\omega_j(w) \geq\threshold$,
\vspace*{-1pt}\cr
0, &\quad \mbox{else}, } %
\]
and $\omega_\infty$ is the pointwise limit of $\omega_j$ as $j\to
\infty$.
The initial configuration $\omega$ is random; $\{\omega(v)\dvtx  v \in V\}
$ is a collection of i.i.d. Bernoulli random variables with parameter
$p$. A natural quantity to study is $\prob_{p}(\omega_\infty\equiv
1)$. Indeed, first results in this area were by van Enter \cite
{vanenter} and Schonmann \cite{schonmann}, who proved that for the
lattice $\mathbb{Z}^d$ this probability is either 1 or 0 according to
whether $\threshold\le d$ or $\threshold> d$. Following the seminal
work of Aizenman and Lebowitz \cite{aizenman},
it became clear that this process is even more interesting on large
\emph{finite} graphs. For a family of graphs depending on a single
parameter $n$, with the number of vertices going to infinity as $n$
increases, we assume that $p=p(n)$, and study the dependence on $n$ of
the critical probability $p_{\mathrm{c}}$ defined by
\[
\prob_{p_{\mathrm{c}}}(\omega_\infty\equiv1)=1/2.
\]

We mention only a few prominent results on how $p_{\mathrm{c}}$ scales with $n$.
Let $[n]=\{1,\dots,n\}$. For a large lattice cube $[n]^d\subseteq
\mathbb{Z}^d$
(where each point is connected to the nearest $2d$ points), Aizenman
and Lebowitz \cite{aizenman} proved that $p_{\mathrm{c}}$ behaves as $(\frac
{1}{\log n})^{d-1}$ when $\threshold=2$, and later Cerf and Cirillo
\cite{cerfcirillo} and Cerf and Manzo\vspace*{-1pt} \cite{cerfmanzo} established
the scaling $(\log_{\threshold-1} n)^{-d+\threshold-1}$ for $3\le
\theta\le d$; here, $\log_{\threshold-1}$ denotes the $(\theta
-1)$st iteration of the logarithm. For the hypercube $\{0,1\}^n$,
Balogh and Bollob\'as \cite{BB:2006} proved that the scaling
for $p_{\mathrm{c}}$ is $n^{-2}4^{-\sqrt n}$ when $\threshold=2$; by contrast,
for the very large threshold
$\threshold= \lceil n/2 \rceil$, the \textit{majority
bootstrap percolation}
studied by Balogh, Bollob\'as and Morris \cite{bollobas},
$p_{\mathrm{c}}$ is close to $1/2$.

Such scaling results do not tell the whole story. They suggest the
existence of an \textit{order parameter}, a function of $p$ and $n$
whose size determines whether $\prob_{p}(\omega_\infty\equiv1)$ is
small or close to 1, for example, on a lattice square $[n]^2$, such a
function is $p\log n$. This leads to two natural questions: Does the
probability exhibit a sharp jump from 0 to 1 as the order parameter
increases? Does the location of the (purported) sharp jump converge as
$n$ increases? (There are good reasons to expect the answer to the
first question to be positive in surprising generality \cite{FK:1996}.)

In a major breakthrough, Holroyd \cite{holroyd} established a positive
answer\vspace*{-1pt} to both questions
in the lattice square case, and proved that $p_{\mathrm{c}}\sim\frac{\pi^2}{18
\log n}$. This celebrated theorem contradicted conjectures based on
simulations, which is due to the fact that $p_{\mathrm{c}}\log n$ converges to its
limit very slowly, as about $1/\sqrt{\log n}$
\cite{GGM:2012}. For lattice cubes $[n]^d$, $d\ge3$ and $2\le\theta
\le d$, the sharp transition was established by Balogh, Bollob\'as,
Duminil-Copin and Morris \cite{bollobas1,bollobas2}.

Besides varying the dimension of the lattice or the threshold, one can
also vary the neighborhood of a point. For example, Holroyd, Liggett
and Romik \cite{hlr} consider the lattice square $[n]^2$, with the
``cross'' neighborhood of a point that consists of $k-1$ points in each
of the 4 axis directions, and $\threshold=k$. In this\vspace*{-1pt} case, $p_{\mathrm{c}}\sim
\frac{\pi^2}{3k(k+1) \log n}$.

In this paper, we consider bootstrap percolation on the \textit
{Hamming torus} (or Hamming graph), the
$d$-fold product graph $K_n\times\cdots\times K_n$, where $K_n$ is
the complete
graph with $n$ vertices.
This graph has vertex set $V = [n]^d$, and two vertices $v\in V$ and
$w\in V$ are adjacent
iff $v-w$ has exactly one nonzero coordinate.
In $d=2$, this graph could be interpreted as taking the
Holroyd--Liggett--Romik neighborhood \cite{hlr} with $k=\infty$. For
any $d$, the neighborhood of a point $v$ is the union of all $d$ lines
through $v$ parallel to the axes. We emphasize, however, that the
threshold $\threshold$ remains fixed as $n$ increases (although some
of our results assume that $\threshold$ is large). Other models of
percolation, including bond percolation \cite{BCHSS:2005,HL:2010} and
site percolation \cite{S:2010}, have been considered on the Hamming
torus, and were shown to exhibit interesting behavior due to the large
neighborhood sizes relative to nearest-neighbor lattices and
hypercubes. For the same reason, we expect qualitatively different
transition phenomena in bootstrap percolation on the Hamming torus from
those described above. First, the critical probability is much smaller.
In fact, our results suggest that
$p_{\mathrm{c}}$ is of the order $n^{-\alpha}$, for some critical exponent
$\alpha> 1$. We are able to determine $\alpha$ exactly in a few
cases, and give estimates otherwise. Moreover, we expect that varying
the order parameter $n^{\alpha}p$ does \textit{not} lead to a sharp
jump of $\prob_{p}(\omega_\infty\equiv1)$ from 0 to 1; instead,
this probability gradually approaches 0 (resp., 1) as the order
parameter approaches 0 (resp., $\infty$). When $d=2$, this is easy to
demonstrate for arbitrary $\threshold$, but when $d\ge3$ the
combinatorics are quite difficult even when $\alpha$ is known exactly.
Nevertheless, we succeeded in analyzing the case $d=\theta=3$, which
has $\alpha=2$: we give an explicit formula for the limit of $\prob
_{p} (\omega_\infty\equiv1)$ when $pn^2=a\in(0,\infty)$. See \cite
{JLTV}, Theorem~3.2, for an analogous result
for bootstrap percolation on Erd\H os--R\`enyi random graphs.

Moreover, in dimensions $d\ge3$, we find two distinct critical
exponents. When $p$ is much smaller than $n^{-1-d/\theta}$, the model
does not accomplish much; with high probability it does not even fill a
single line. When $p$ is much larger than $n^{-1-d/\theta}$, but
smaller than\vspace*{-1pt} $n^{-1-2/\theta-c'/\threshold^{3/2}}$, for
large enough $\theta$, with high probability some lines become open,
but no two-dimensional subgraphs do, and thus\vspace*{-1pt} $\prob_{p}(\omega
_\infty\equiv1)\to0$. When $p>n^{-1-2/\theta-c''/\threshold
^{3/2}}$, and $\threshold$ is large
enough,
$\prob_{p}(\omega_\infty\equiv1)\to1$. Here, $0<c''<c'$ are
constants depending on $d$.

It remains an open question for $\threshold>2$ whether the critical
exponents for the appearance of open subspaces with dimension $i$ are
distinct for each $2\leq i \leq d$. However, in subsequent work,
Slivken has proven that for $\threshold=2$, there are distinct
critical exponents for the appearance of open subspaces with dimension
$2i$ for $1\le i < \sqrt{d}$ \cite{slivken}.


\section{Statement of results}


Let $\cF$ be a family of subsets of $[n]^d$. Then
\[
\prob_p(\exists F \in\cF\dvtx \omega_\infty|_F
\equiv1)
\]
is a nondecreasing function in $p$. (Observe that here the vertical bar
does not denote
a conditional probability but a restriction, i.e., $\omega_\infty|_F$
is $\omega_\infty$ restricted
to the set $F\subseteq[n]^d$.) For $\cF_i$, the collection of
$i$-dimensional subgraphs of $G$, there exists a threshold function
$p_{\mathrm{c}}(i,d)$ such that
\[
\prob_{p_{\mathrm{c}}(i,d)}(\exists F \in\cF_i\dvtx
\omega_\infty|_F \equiv1)=0.5.
\]
If $\omega_j(v)=1$, we say $v$ is open at step $j$, and a set $S
\subseteq V$ is open if each $v \in S$ is open, that is, $\omega
_j|_S\equiv1$.



For $i=0$, we have an additional critical probability
$p^*_{\mathrm{c}}(0,d)$. We would like to define it to be the threshold function
for the event that
$\omega_\infty\not\equiv\omega_0$; unfortunately, this is not an increasing
event. (Recall that an event $E\subseteq\{0,1\}^{V}$
is increasing if $\omega\in E$ and $\omega\le\omega'$
together imply $\omega'\in E$.) Instead, we define the event
\[
\abovethresh= \biggl\{ \exists v\dvtx \sum_{w \sim v}
\omega_0(w)\geq\threshold \biggr\}
\]
and $p^*_{\mathrm{c}}(0,d)$ to be the $p$ for which $P_p(\abovethresh)=0.5$.

We write $f(n)\sim g(n)$ if $\frac{f(n)}{g(n)} \to1$ as $n\to\infty$.
We conjecture that for every $\threshold, i, d \in\N$ with $i \leq
d$, there exists $a_{\mathrm{c}}=a_{\mathrm{c}}(\theta,i,d)$ and $\alpha_{\mathrm{c}}=\alpha_{\mathrm{c}}(\theta,i,d)$
such that
\[
p_{\mathrm{c}}(i,d) \sim a_{\mathrm{c}}n^{-\alpha_{\mathrm{c}}}.
\]
Moreover, there exists a nondecreasing function $G=G(\theta,i,d)\dvtx \R^+
\to[0,1]$ such that
$G(x) \to0$ as $x \to0$,
$G(x) \to1$ as $x \to\infty$, and if $p=a n^{-\alpha_{\mathrm{c}}}$ then
\[
\prob_{p}(\exists F \in\cF_i\dvtx
\omega_\infty|_F \equiv1) \sim G(a).
\]

We are able to prove that this is the case for $d=2$.

%
%
\begin{theorem}
\label{2d-thm}
Let $d=2$, $k = \lceil\threshold/2 \rceil> 1$ and $p=an^{-1-\sfrac{1}{k}}$.
Then
\[
\prob( \omega_\infty\equiv1) \to%
\cases{ 1-e^{-2a^k/k!},&
\quad \mbox{if} $\threshold$ \mbox{is odd},
\cr
\bigl(1-e^{-a^k/k!}
\bigr)^2, &\quad \mbox{if} $\threshold$ \mbox{is even}. } %
\]
Thus,
\[
p_{\mathrm{c}}(2,2)=p_{\mathrm{c}}(1,2)=p^*_{\mathrm{c}}(0,2)=n^{-1-\sfrac{2}{\threshold} +
o(\threshold^{-3/2})}.
\]
Furthermore,
\[
\prob \bigl(\{\omega_\infty\not\equiv\omega_0\} \setminus
\{\omega_\infty\equiv1\} \bigr)=o(1).
\]
\end{theorem}

As $d$ increases the problem becomes more intricate. For $d=3$, we are
able to identify the limit under critical scaling when
$\threshold=3$.

%
%
\begin{theorem}
\label{3d-spanning-thm}
Let $d=3$, $\threshold=3$ and $p = an^{-2}$ with $a>0$. Then as $n\to
\infty$
%
\begin{eqnarray}
&&\prob_{p}(\omega_\infty\equiv1) \to1 - e^{-a^3 -
(3/2)a^2(1-e^{-2a})}
\nonumber
\\[-8pt]
\\[-8pt]
&&\hphantom{\prob_{p}(\omega_\infty\equiv1) \to1 -} {}
\times \biggl[\frac{3}{2} a^2 \bigl( \bigl(e^{-a}+ae^{-3a}
\bigr)^2-e^{-2a} \bigr)e^{-a^2e^{-2a}} + e^{a^3e^{-3a}}
\biggr].
\nonumber
\end{eqnarray}
\end{theorem}


Other three-dimensional results include
determining the critical exponents $(\alpha_{\mathrm{c}})$ for $d=3$ and low
thresholds, but not the exact constants $a_{\mathrm{c}}$;
see Section~\ref{suffforplanes} for details.


Observe the contrast between Theorem~\ref{2d-thm} and Theorem~\ref
{3d-spanning-thm}
and classical results on percolation on lattice cubes $[n]^d$ \cite
{holroyd,bollobas2}:
not only is the critical scaling $p=an^{-\alpha}$ much smaller in the
present case, but also
$\lim_n \prob_{p}(\omega_\infty\equiv1)$ is not a step function of
$a$. Instead, this
limiting probability varies continuously from $0$ to $1$ as $a$
increases from $0$ to $\infty$.

Many of our results state that
\[
p_{\mathrm{c}}(i,d)=n^{-1 - \sfrac{c_1(i,d)}{\threshold} - \Theta(\threshold^{-3/2})},
\]
where $c_1=c_1(i,d)$ is a constant. This shorthand notation means that,
for a large
$n$, we can get a lower bound and  upper bound for $p_{\mathrm{c}}(i,d)$ of the
stated form,
with constants in the correction term $\Theta(\threshold^{-3/2})$
depending on $i$ and $d$.

For general $d\ge3$,
we calculate $p_{\mathrm{c}}^*(0,d)$ and $p_{\mathrm{c}}(1,d)$ for all $d\geq2$ quite
precisely.

%
%
\begin{theorem} \label{la liga}
Let $p=f(n)n^{-1-\sfrac{d}{\threshold}}$ and $d,\threshold\geq3$.
If $f(n)\to0$ then
\[
\prob(\abovethresh) \to0
\]
and
if $f(n)\to\infty$ then
\[
\prob(\exists\mbox{ a line $\Ln$ such that } \omega_\infty|_\Ln
\equiv1) \to1.
\]
%
\end{theorem}

Furthermore, we get good bounds on
$p_{\mathrm{c}}(2,d)$, the threshold for existence of two-dimensional subspaces in
the final configuration.

%
%
\begin{theorem} \label{epl} Fix $d$ and fix $\threshold$ sufficiently
large depending on $d$. For $n$ sufficiently large,
\[
n^{-1 - \sfrac{2}{\threshold} - \vfrac{4d^2 + 3}{\threshold^{3/2}}} \leq p_{\mathrm{c}}(2,d) \leq n^{-1 - \sfrac{2}{\threshold} - \sfrac{\sqrt
{8(d-2.1)}}{\threshold^{3/2}}}.
\]
\end{theorem}

[We have not attempted to optimize the constants $\sqrt{8(d-2.1)}$ and
$4d^2 + 3$ in the above theorem.]
The key arguments in the proof of Theorem~\ref{epl} are Lemmas \ref
{cajunsprel1} and \ref{2d-span-lb-lem}.

The higher the dimensions $i$ and $d$,
the more difficult it becomes to calculate $p_{\mathrm{c}}(i,d)$.
However, Theorems \ref{la liga} and \ref{epl} are sufficient for us
to get bounds on $p_{\mathrm{c}}(i,d)$ for all $i,d \geq2$.

%
\begin{theorem}
For all $i\ge2$ and $d$, and sufficiently large $n$,
\[
p_{\mathrm{c}}(i,d)=n^{-1 - \sfrac{2}{\threshold} - \Theta(\threshold^{-3/2})}.
\]
\end{theorem}

\begin{pf}
It is easy to see that
$p_{\mathrm{c}}(i,d)$ is nondecreasing in $i$ and decreasing in $d$. Also
$p_{\mathrm{c}}(d,d)$ is decreasing in $d$. To see this last inequality note that
when $n \geq3 \threshold$ and $d=j+1$
\[
\prob_{ p_{\mathrm{c}}(j,j)} (\exists\mbox{ at least } \threshold\ i \mbox{ such that $
\omega_\infty|_{(i,*,*, \ldots)}\equiv1$})>1/2.
\]
The event on the left-hand side implies that $\omega_\infty\equiv1$, and
thus
\[
p_{\mathrm{c}}(j+1,j+1)\leq p_{\mathrm{c}}(j,j)
\]
and inductively
\[
p_{\mathrm{c}}(d,d) \leq p_{\mathrm{c}}(3,3).
\]
So
\[
p_{\mathrm{c}}(2,d) \leq p_{\mathrm{c}}(i,d) \leq p_{\mathrm{c}}(d,d)
\leq p_{\mathrm{c}}(3,3).
\]
By Theorem~\ref{epl},
\[
p_{\mathrm{c}}(2,3) \leq n^{-1 - \sfrac{2}{\threshold} - \vfrac{\sqrt
{7.2}+o(1)}{\threshold^{3/2}}}.
\]
By coupling it is easy to see that $\omega$ chosen when $p=10
\threshold p_{\mathrm{c}}(2,3)$ stochastically dominates the union of $10
\threshold$ independent $\omega'$
chosen with $p=p_{\mathrm{c}}(2,3)$. Then by the definition of $p_{\mathrm{c}}(2,3)$
\[
\prob_{10 \threshold p_{\mathrm{c}}(2,3)} (\exists\mbox{ at least } \threshold\ i \mbox{ such that $
\omega_\infty|_{(i,*,*)}\equiv1$})>1/2.
\]
The event on the left-hand side implies $\omega|_{\infty}\equiv1$,
and thus
%
\begin{equation}
\label{cloudforestcafe} p_{\mathrm{c}}(3,3) \leq10 \threshold p_{\mathrm{c}}(2,3).
\end{equation}
And putting this all together for all $d \geq3$ and $2 \leq i \leq d$,
\begin{eqnarray*}
n^{-1 - \sfrac{2}{\threshold} - \vfrac{4d^2+2+o(1)}{\threshold
^{3/2}}} &\leq& p_{\mathrm{c}}(2,d) \leq p_{\mathrm{c}}(i,d) \leq
p_{\mathrm{c}}(d,d) \leq p_{\mathrm{c}}(3,3)
\\
&\leq&10 \threshold p_{\mathrm{c}}(2,3) \leq n^{-1 - \sfrac{2}{\threshold} - \vfrac{\sqrt
{7.2}-o(1)}{\threshold^{3/2}}},
\end{eqnarray*}
which is the desired result.
\end{pf}

%
%
\begin{remark} The above results are all asymptotic statements in $n$.
One natural question is whether we can obtain nonasymptotic bounds on
the critical parameters. Our arguments do in fact produce bounds on the
critical probability for specific values of $n$. Keeping track of (or
even stating) these bounds is quite challenging and we have made no
attempt to optimize them. Different results kick in at different values
of $n$, but all of them
work if $n$ is at least roughly $e^{\threshold^{3/2}}$.
\end{remark}


The rest of the paper is organized as follows. In Section~\ref{2d}, we
prove the two-dimensional Theorem~\ref{2d-thm}.
In Section~\ref{3d-upper}, we give a necessary condition for a plane
to become
open when $d=3$ and in Section~\ref{suffforplanes} we give a
sufficient condition
for this event for arbitrary $d$. Section~\ref{suffforplanes} also
features the resulting upper
and lower bounds for critical exponents in three dimensions and the
proof for the upper bound in Theorem~\ref{epl}. Section~\ref{3d-precise}
features the proof of Theorem~\ref{3d-spanning-thm}, which
is, like that of Theorem~\ref{2d-thm}, based on Poisson convergence.
While the two-dimensional case
requires nothing more than Poisson approximation to the binomial, our
proof of this three-dimensional result hinges on much more intricate coupling
methods introduced in \cite{poissonbook}. As some events in question
are not
positively related, the required couplings need to be explicitly constructed;
the details of this construction are deferred to the \hyperref[ap:poisson]{Appendix}.
In Section~\ref{lines}, we study when a line is likely to become open
and establish Theorem~\ref{la liga}. In Section~\ref{nessforplanes},
we provide a lower bound on the value of $p$ that makes
it likely that a plane becomes open; this, together with results in
Section~\ref{suffforplanes},
will complete the proof of Theorem~\ref{epl}.
We conclude with a short list of open questions in Section~\ref{open}.


We end this section with a note on terminology, adopted from \cite{aizenman}.
A vertex $v$ (resp., a set $F\subset[n]^d$)
is called \textit{open}, or \textit{occupied} at a time $t\in
[0,\infty]$ if
$\omega_t(v)=1$ (resp., $\omega_t|_F\equiv1$).
Assume $G\subset[n]^d$ is an arbitrary (deterministic or random) set, and
suppose the bootstrap percolation process is run started from the set
of open vertices
equal to $G$. Fix also a set $F\subset[n]^d$.
We say that $G$ \textit{spans} $F$ if this process makes every vertex
in $F$
eventually open. Furthermore, we say that $F$ is \textit{internally
spanned} by $G$
if $G\cap F$ spans $F$.
When $F$ is unspecified, it is assumed to be the entire
torus $[n]^d$. As throughout this section, the initially
open points are by default chosen at random, independently with
probability $p$; if this
set spans, we also say
that \textit{spanning} occurs. Finally, we denote by
$\sigma_\theta(d,p)$ the \textit{spanning probability}, that is, the
probability
of spanning for the $d$-dimensional torus with threshold $\theta$ and initial
occupation density $p$. (Note that
the dependence on $n$ is suppressed in this notation.)

\section{Precise two-dimensional results}
\label{2d}
In the two-dimensional case, we can describe the limiting behavior
exactly as $n \to\infty$. Let $k = \lceil\threshold/2\rceil$ and
$p = an^{-1-\sfrac{1}{k}}$ for some constant $a$. Also assume $k > 1$;
the cases $\threshold=1$ and $\threshold=2$ are easy to work out
separately. (For $\threshold=1$, $\omega_\infty\equiv1$ if and only
if $\omega_0 \not\equiv0$; for $\threshold=2$, $\omega_\infty
\equiv1$ asymptotically if and only if $\omega_0$ contains at least
two noncollinear open points.)

%
\begin{lemma} \label{overlines}
Let $k = \lceil\threshold/2\rceil$ and $p = an^{-1-\sfrac{1}{k}}$.
With probability going to 1, there are no lines with at least $k+1$
points initially open.
\end{lemma}

\begin{pf}
For a fixed line $\ell$, let $E_\ell$ be the event that
$\ell$ contains $k+1$ initially open points. For any $\ell$,
\[
\prob_{p}(E_\ell)\le{\binom{n} {k+1}}p^{k+1} \le
n^{k+1}p^{k+1}\le a^{k+1}n^{-1-\sfrac{1}{k}},
\]
and, as there are $2n$ lines,
\[
\prob_{p}\biggl(\bigcup_\ell
E_\ell\biggr)\le2n\cdot a^{k+1}n^{-1-\sfrac
{1}{k}}=2a^{k+1}n^{-\sfrac{1}{k}}
\to0
\]
as $n\to\infty$.
\end{pf}

%
\begin{lemma} \label{underlines} Fix an $\epsilon>0$.
Let $k = \lceil\threshold/2\rceil$ and $p = \epsilon n^{-1-\sfrac
{1}{k}}$. Fix constants $A,B$
and choose $B$ fixed vertical (resp. horizontal) exceptional lines.
With probability going to 1, there are at least $A$ horizontal (resp., vertical)
lines, which contain $k-1$ initially open points none of
which are in the union of the exceptional lines.
\end{lemma}

\begin{pf}
Each of the $n$ horizontal lines satisfies the condition independently
with probability at least
\[
{\binom{n-B} {k-1}}p^{k-1}(1-p)^{n-k+1} = \Theta
\bigl(n^{-1+\sfrac{1}{k}} \bigr).
\]
The probability that there are at least $A$ such lines
therefore goes to 1.
\end{pf}

Let $E_{\horizontal}$ be the event that some horizontal line contains
at least
$k$ initially open points, $E_{\vertical}$ the corresponding event for
vertical lines, and
$E_{\horizontal} \circ E_{\vertical}$ the event that the two occur disjointly.

%
\begin{lemma} \label{e1e2}
Let $k = \lceil\threshold/2\rceil$ and $p = a n^{-1-\sfrac{1}{k}}$.
We have
\[
\prob_p\bigl((E_{\horizontal} \cap E_{\vertical}) \setminus
(E_{\horizontal} \circ E_{\vertical})\bigr) \to0.
\]
Furthermore,
\[
\prob_p(E_{\horizontal} \cap E_{\vertical}) \to
\bigl(1-e^{-a^k/k!} \bigr)^2
\]
and
\[
\prob_p(E_{\horizontal} \cup E_{\vertical}) \to1-
\bigl(e^{-a^k/k!} \bigr)^2.
\]
\end{lemma}

\begin{pf}
The event $(E_{\horizontal} \cap E_{\vertical}) \setminus
(E_{\horizontal} \circ E_{\vertical})$ happens only if some point $v$
is open, and each of the two lines through $v$ contains exactly $k-1$
additional open points.
The probability that such a point exists is bounded by
\[
n^2p \bigl({n^ {k-1}}p^{k-1} \bigr)^2 =
O \bigl(n^{-1+\sfrac{1}{k}} \bigr) \to0.
\]
This
proves the first assertion.\vadjust{\goodbreak}

As $E_{\horizontal}$ and $E_{\vertical}$ are increasing events,
$\prob_p(E_{\horizontal} \cap E_{\vertical})\ge\prob
_p(E_{\horizontal})\times\allowbreak \prob_p(E_{\vertical})=\prob_p(E_{\horizontal})^2$
by the FKG inequality.
Conversely, the BK inequality gives
$\prob_p(E_{\horizontal})\prob_p(E_{\vertical}) \ge\prob
_p(E_{\horizontal} \circ E_{\vertical})$.
Thus, $\prob_p(E_{\horizontal} \cap E_{\vertical})- \prob
_p(E_{\horizontal})^2\to0$. Moreover,
the number of horizontal lines with at least $k$ open points is
Binomial and converges
in distribution to
a Poisson random variable with expectation $a^k/k!$. Thus, $\prob
_p(E_{\horizontal})\to1-e^{-a^k/k!}$,
which easily ends the proof.
\end{pf}

Let $G$ be the event that the entire graph becomes open, that is, $G =
\{\omega_\infty\equiv1\}$.

%
\begin{lemma} \label{events}
Let $k = \lceil\threshold/2\rceil$ and $p = a n^{-1-\sfrac{1}{k}}$.
If $\threshold$ is even, $\prob_p(G)-P(E_{\horizontal} \cap
E_{\vertical}) \to0$, while if $\threshold$ is odd,
$\prob_p(G)- \prob_p(E_{\horizontal} \cup E_{\vertical}) \to0$.
\end{lemma}

\begin{pf}
If $\threshold$ is odd, the process adds no new open vertex unless
there is some line with at least $k$ vertices initially open. So $G
\subseteq E_{\horizontal} \cup E_{\vertical}$.
If $\threshold$ is even, then by Lemma~\ref{overlines},
$\prob_p(G \setminus(E_{\horizontal} \cap E_{\vertical})) \to0$.

Fix an $\epsilon>0$ and let $\omega^*$, $\omega'$ and $\omega''$
be three
independent configurations, the first with $p^*=(1-2\epsilon)
n^{-1-\sfrac{1}{k}}$, and
the other two are ``sprinkled''
with small $p'=\epsilon n^{-1-\sfrac{1}k}$. Observe that $\omega_0$
(generated with $p$) stochastically
dominates $\omega^*\cup\omega'\cup\omega''$.

Now suppose $\threshold$ is odd and $E_{\horizontal} \cup
E_{\vertical}$ occurs in $\omega^*$.
Then some line $\ell$ has $k$ points open in $\omega^*$.
We now describe the events that occur with probability 1 as $n\to
\infty$.
By Lemma~\ref{underlines},
there are $\threshold$ lines $\{\ell_i'\}$ parallel to $\ell$, each
with $k-1$
points open in $\omega'$.
Moreover, again by Lemma~\ref{underlines}, there are $\threshold$
lines $\{\ell_j''\}$
perpendicular to $\ell$,
each with $k-1$ points, which are open in $\omega''$ and avoid $\ell$
and all
$\ell_i'$.

Let $G^*$ be the event that the initial configuration
$\omega^*\cup\omega'\cup\omega''$ eventually causes every point to
be open.
We claim that if the events in the above paragraph all happen then
$G^*$ happens.
First, each point of intersection of $\ell_j''$ and $\ell$ becomes
open as it sees
$k-1$ open neighbors on $\ell_j''$ and $k$ on $\ell$. Then there are
$\threshold$ open points on
$\ell$, so $\ell$ becomes open.
Now each point of intersection of $\ell_j''$ and $\ell_i'$ becomes
open as it
sees one open neighbor on $\ell$, and $k-1$ additional
open neighbors each on $\ell_j''$ and $\ell_i'$.
This results in $\threshold$ open points on each $\ell_i''$ and $\ell
_i'$, so
these $2\threshold$ lines all become open, and the entire graph
becomes open in the next step.

It now follows that $\liminf\prob_p(G)\ge\liminf\prob
_{p^*}(E_{\horizontal} \cup E_{\vertical})$,
and the claim for odd $\threshold$ follows by continuity (in $a$) of
limits in Lemma~\ref{e1e2}.

Now suppose $\threshold$ is even. If $E_{\horizontal} \cap
E_{\vertical}$ occurs, then
we may assume $E_{\horizontal} \circ E_{\vertical}$ occurs by
Lemma~\ref{e1e2}.
That is, there is a horizontal line $\ell_h$ and a vertical line $\ell_v$,
each with $k$ points initially open, excluding their point of intersection.
This point of intersection becomes open at the first time step.

As in the odd case, we may use sprinkling and Lemma~\ref{underlines}
to produce
$\threshold$ horizontal lines $\ell_i'$ and $\threshold$ vertical
$\ell_j''$, each with $k-1$ initially
open points that avoid all other lines.
Then every point of intersection between $\ell_h$ and $\ell_j''$, and
between $\ell_v$ and
$\ell_i'$, sees $\threshold= (k+1)+(k-1)$ open sites, so it becomes open.
Then $\ell_h$ and $\ell_v$ contain $\threshold$ open sites, so they
become open.
Then every point of intersection of an $\ell_i'$ with an $\ell_j''$ sees
$2 + 2(k-1) = \threshold$ open sites, so becomes open. Now
the entire graph becomes open in two additional steps.
\end{pf}

\comment{
%
%
\begin{theorem}
Let $\threshold\ge3$, $p=an^{-1-\sfrac{1}{k}}$ where $k = \lceil
\sfrac{\threshold}{2}\rceil$. As $n \to\infty$,
\begin{eqnarray*}
\lim_{n \to\infty}\prob(F)= %
\cases{ \bigl(1-e^{-a^k/k!}
\bigr)^2 & \mbox{if} $\threshold$ \mbox{is even},
\cr
1-
\bigl(e^{-a^k/k!} \bigr)^2 &\mbox{if} $\threshold$ \mbox{is odd}.
} %
\end{eqnarray*}
%
\end{theorem}
}

\begin{pf*}{Proof of Theorem~\ref{2d-thm}}
The claimed convergence follows from Lemmas \ref{e1e2} and \ref{events}.
\end{pf*}

\section{Upper bound on critical exponent in three dimensions}
\label{3d-upper}
It is easy to see that with $p = n^{-\alpha}$ for $\alpha> 1 + \frac
{d}{\threshold}$, with high probability, no points that are not
initially open become open. [The expected number of vertices with at
least $\threshold$ open neighbors is at most $Cn^d (np)^\threshold=
O(n^{d+\threshold- \alpha\threshold}) = o(1)$.]
In this section, we will assume that $d=3$ and $\theta\ge3$ and
establish a bound on $\alpha$ that
ensures that no planes become open (and hence the entire Hamming torus
does not become open)
with high probability. A similar
result is proved for general $d$ in Section~\ref{nessforplanes}.

%
\begin{lemma}\label{d3upperbound}
Let $d=3$ and $\theta=2k-1 \ge3$ be odd. Let $p = n^{-\alpha}$ for
$\alpha> 1 + \frac{8}{3\theta-1}$. Then $\prob_p($a plane becomes
open$)\to0$.
The same holds for $\theta=2k \ge4$ when $\alpha> 1 + \frac
{8}{3\theta-2}$.
\end{lemma}

\begin{pf}
We may assume $\threshold\geq4$, since the $\threshold= 3$ bound of
$\alpha> 2$ is equivalent to $\alpha>1 + \frac{d}{\threshold}$. We
will prove the lemma for $\threshold$ odd; the even case is similar.
Define the following three conditions for a vertex $v$:
\begin{longlist}[(3)]
\item[(1)]$v$ is initially open,
\item[(2)]$v$ is on a line with at least $k$ points initially open,
\item[(3)] the neighborhood of $v$ has at least $\theta$ points initially open.
\end{longlist}

We first prove
%
\begin{equation}
\label{3dstep1} \hspace*{20pt}\prob_p \bigl(\mbox{there exists a plane each of
whose points satisfies one of (1)--(3)} \bigr)\to0
\end{equation}
To prove (\ref{3dstep1}), we fix a plane $P$, which we may assume to
be the $e_1,e_2$-plane,
and prove that the probability that
all of its points satisfy one of (1)--(3)
is exponentially small. Fix an $\epsilon\in(0,1/3)$.
Consider the lines perpendicular to $P$, horizontal lines
in $P$, and vertical lines in $P$, that contain at least one initially
open point.
Let their respective numbers be $S_1$, $S_2$ and $S_3$, and note that
each of these
three numbers is Binomially distributed. The probability that a fixed line
contains an initially open vertex is at most $np=o(1)$, so $\prob
_p(S_1\ge\epsilon n^2)$,
$\prob_p(S_2\ge\epsilon n)$, and
$\prob_p(S_3\ge\epsilon n)$ are all exponentially small. With\vadjust{\goodbreak}
probability exponentially
close to 1, the number of points in $P$ included in one of the three
types of lines is therefore
at most $3\epsilon n^2$, which proves (\ref{3dstep1}).

Let $E_v$ be the event that the point $v$ violates all three conditions
(1)--(3),
but that it becomes open and that no point violating these conditions becomes
open earlier. It remains to show that
%
\begin{equation}
\label{3dstep2} \prob_p(E_v)=o \bigl(1/n^3
\bigr).
\end{equation}
We will denote by $\mathcal{N}(v)$ the neighborhood of a point $v$.
If $E_v$ occurs, then $\mathcal{N}(v)$ has $m$ points initially open,
for some
$0 \le m \le\theta- 1$. Then $\mathcal{N}(v)$ contains $\theta-m$
other points
$w_1, \ldots, w_{\theta-m}$, not initially open, which become open
before $v$.
Thus, these $w_i$ must satisfy (2) or (3).
Because $v$ violates (2), each $w_i$ shares with $v$ at most $k-1$
initially open neighbors.
Therefore, whether $w_i$ satisfies (2) or (3),
$\mathcal N_i=\mathcal{N}(w_i) \setminus\mathcal{N}(v)$ must contain
$k$ initially open points.

Assume $m$ and $w_i$ are selected. Let $N$ be the number of
initially open points in $\mathcal{N}_i\cap\mathcal{N}_j$, for some
$i\ne j$. (Note that the intersection of three or more $\mathcal{N}_i$
is empty.)
Let $H_b^m$ be the event that $\mathcal{N}(v)$ has $m$ initially open
points, $w_1, \ldots, w_{\threshold-m}$ exist such that $\mathcal{N}_i$
all contain $k$ initially open points \textit{and} that
$N=b$. Then
%
\begin{equation}
\label{3dstep3} P \bigl(H_0^m \bigr)\le
C(np)^m n^{\theta-m} \bigl((np)^k
\bigr)^{\theta-m}
\end{equation}
for some constant $C$. To estimate $P(H_b^m)$, observe that
each increase of $b$ by 1 contributes an additional factor of $p$ and
removes a factor
$(np)^2$ from the right-hand side of (\ref{3dstep3}). By monotonicty,
we may assume
$\alpha\le2$ so $p\le(np)^2$ [recall $\threshold\geq5$ so $1+
8/(3\threshold-1)<2$]; then $P(H_b^m)\le P(H_0^m)$ for all $b\ge0$
and $m$.
Furthermore, $n^k p^{k-1} = o(1)$ (since $k\geq2$), thus the upper
bound in (\ref{3dstep3}) increases
with $m$. It follows that $P(E_v)$ is bounded by the expression in
(\ref{3dstep3}) with $m=\theta-1$,
which gives
\[
n^3P(E_v)\le C n^{3k+2}p^{3k-2}\to0,
\]
proving (\ref{3dstep2}).
\end{pf}

\section{Internally spanned planes}
\label{suffforplanes}
In this section, we prove the upper bound in Theorem~\ref{epl}
regarding $p_{\mathrm{c}}(2,d)$, the critical probability for the existence of
two-dimensional planes in the final configuration. We also introduce a
dimension-reduction inequality that allows us to compute lower bounds
on the spanning probabilities
$\sigma_\theta(\theta,p)$ for arbitrary $d$ and $\threshold$.
Our first result is a lower bound on $\sigma_\theta(2,p)$, which will
allow us to find lower bounds for all $d$ later on.

%
%
\begin{figure}[b]

\includegraphics{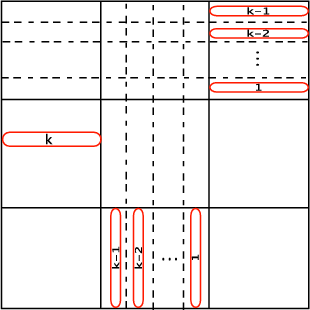}

\caption{This configuration will span the two-dimensional Hamming
graph when $\theta= 2k-1$ is odd. Each region bounded by solid lines
is approximately $n/3 \times n/3$. The hashed lines are spaced $\frac
{n}{3(k-2)}$ units apart, so each subregion has height and width on the
order of $n$. A red oval represents the existence of at least one line
(in the direction indicated) in that region with the specified number
of open vertices.}
\label{ht2d-any-p}
\end{figure}

%
\begin{lemma}
\label{2d-span-lb-lem}
Let $k = \lceil\theta/2 \rceil$ and $\liminf n^{\alpha
}p = b >0$ with
$\alpha> 1+1/k$. Then there exists a constant $C>0$ depending on
$\theta$ and $b$ such that for all sufficiently large $n$, $\sigma
_{\theta}(2,p)\geq Cn^{-\beta}$ where
%
\begin{equation}
\beta(\alpha) = %
\cases{ \alpha k^2 +a(a+1) - \alpha
a(a-1) - (k+1)^2, &\quad $\theta$ \mbox{odd},
\cr
\alpha k(k+1) +
a(a+1) - \alpha a(a-1) - (k+1) (k+2), &\quad $\theta$ \mbox{even}, }\hspace*{-38pt} %
\end{equation}
and $a = \lfloor\alpha/(\alpha-1) \rfloor$.
\end{lemma}

%
\begin{remark} If $\alpha= 1+1/k$ and $p = b/n^\alpha$ then $\sigma
_\theta(2,p) \to c \in(0,1)$ by Theorem~\ref{2d-thm}, so $\beta
(\alpha) = 0$ for $\alpha\leq1+1/k$.
\end{remark}

\begin{pf*}{Proof of Lemma~\ref{2d-span-lb-lem}}
Observe that the configuration in Figure~\ref{ht2d-any-p} is
sufficient for spanning for odd $\theta= 2k-1$. In the figure, the
two-dimensional Hamming graph is first subdivided into nine regions
that have dimensions $n/3 \times n/3$. The hashed lines further
subdivide some of the regions, and are spaced $\frac{n}{3(k-2)}$ units
apart, so each subregion has height and width on the order of $n$. Each
red oval represents the existence of at least one line (in the
direction indicated) in that region with the specified number of open
vertices. To check that this configuration leads to spanning, observe
that the horizontal line containing $k$ open vertices is the first to
be spanned: after one step the vertex at the intersection of this line
and the vertical line with $k-1$ open vertices becomes open, after two
steps the vertex at the intersection of this line and the vertical line
with $k-2$ open vertices becomes open, and so on until this line
contains $2k-1$ open vertices and the entire line becomes open. As this
line is made open, all of the vertical lines each gain one additional
open vertex, so the vertical line with $k-1$ initially open vertices is
next to be
spanned in the same fashion, followed by the horizontal line with $k-1$
open vertices and so on until all $2k-1$ lines with ovals are spanned
and cause the rest of the graph to become open. The reason for
subdividing the graph into disjoint regions like we have is so that all
of the events depicted are independent. Therefore, the spanning
probability is bounded below as
%
\begin{eqnarray}
\label{2d-spanning-lowerbd} %
\sigma_{2k-1}(2,p) &\geq&
\mathbb{P}_{p} (\mbox{configuration in Figure~\ref{ht2d-any-p}} )
\nonumber
\\
&=& \biggl[ 1 - \biggl( 1 - \frac{1}{k!} n^k p^k +
o \bigl((np)^k \bigr) \biggr)^{n/3} \biggr]
\\
&&{}\times\prod_{\ell= 1}^{k-1} \biggl[1 -
\biggl(1 - \frac{1}{\ell!}(np)^{\ell} + o \bigl((np)^{\ell}
\bigr) \biggr)^{n/3(k-2)} \biggr]^2.
\nonumber
\end{eqnarray}
If $p \asymp n^{-\alpha}$ and $\alpha< 1 + \frac{1}{k}$ then the
lower bound in (\ref{2d-spanning-lowerbd}) tends to $1$ as $n\to
\infty$,\vspace*{-1pt} in agreement with Theorem~\ref{2d-thm}, so we assume $p
\asymp n^{-\alpha}$ and $\alpha> 1+\frac{1}{k}$. In this case, the
terms in the product in the last line of (\ref{2d-spanning-lowerbd})
for which $\ell\leq1/(\alpha-1)$ either tend to $1$ or (in the case
of equality) are bounded away from $0$ as $n \to\infty$. Therefore,
by applying the bound $(1-x)^m \leq1 - mx + m^2 x^2$ for $x\in(0,1)$,
we bound (\ref{2d-spanning-lowerbd}) from below by
%
\begin{equation}
\label{2d-spanlb} C \bigl[n^{k+1}p^k - o \bigl(n^{k+1}p^k
\bigr) \bigr] \prod_{\ell= a}^{k-1} \bigl[
n^{\ell+ 1}p^{\ell} - o \bigl(n^{\ell+1}p^\ell
\bigr) \bigr]^2,
\end{equation}
where $a = \lfloor\alpha/ (\alpha-1) \rfloor$ and the
value of $C$ here is
not smaller than $(3\cdot k!)^{-2k}$ for any $\alpha>1+1/k$. We can
take $p = (b/2)n^{-\alpha}$ by noting that $\sigma_\threshold(2,p)$
is increasing in $p$, so the constant $C$ appearing in the lemma is not
smaller than $(3\cdot k!)^{-2k}(b/2)^{k(k+1)}$. Computing the exponent
of the leading order term in (\ref{2d-spanlb}) when $p =
(b/2)n^{-\alpha}$ gives the formula for $\beta(\alpha)$ when $\theta
$ is odd. A configuration similar to the one in Figure~\ref{ht2d-any-p}, but where there is one additional column with $k$
initially open vertices, provides a sufficient condition for spanning
when $\theta= 2k$. This leads to an expression like the one in (\ref
{2d-spanning-lowerbd}), except with the first factor squared, and leads
to the formula for $\beta(\alpha)$ when $\theta$ is even.
\end{pf*}

Our first application of Lemma~\ref{2d-span-lb-lem} is to prove the
upper bound in Theorem~\ref{epl}.

%
\begin{theorem}
\label{2d-largetheta-upper-thm}
Fix $d\geq3$ and fix $\theta$ large enough depending on $d$ [$\theta
\geq650(d-2.1)$ is sufficient]. For all sufficiently large $n$,
\[
p_{\mathrm{c}}(2,d) \leq n^{-1 - \sfrac{2}{\theta} - \sqrt{8(d-2.1)}/\theta^{3/2}}.
\]
\end{theorem}

To prepare for the proof, we need a bound on the function $\beta
(\alpha)$ in Lemma~\ref{2d-span-lb-lem} that eliminates the use of
the floor function. We isolate the reasoning by treating just the terms
involving $a$.

%
\begin{lemma}
If $1<\alpha\leq2$ and $a = \lfloor\alpha/(\alpha-1)
\rfloor$ then
%
\begin{equation}
\label{abound} a(a+1) - \alpha a(a-1) \leq\frac{1}{\alpha- 1} + 1 +
\frac
{1}{2}(\alpha-1).
\end{equation}
\end{lemma}

\begin{pf}
Let $\epsilon= \alpha- 1$ and suppose $\frac{1}{\epsilon} = m + u$
where $m \geq1$ is an integer and $u \in[0,1)$. Then we can write
(\ref{abound}) as
\[
a(-\epsilon a +2 + \epsilon) - \frac{1}{\epsilon} \leq1+ \frac{1}2
\epsilon,
\]
so we must prove this inequality. Observe that
\[
a = \biggl\lfloor\frac{1 + \epsilon}{\epsilon} \biggr\rfloor= \lfloor m + u + 1 \rfloor=
m+1,
\]
so we have
\begin{eqnarray*}
a(-\epsilon a +2 + \epsilon) - \frac{1}{\epsilon} &=& \frac{ -
(m+1)^2 + 2(m+u)(m+1) + m+1 - (m+u)^2}{m+u}
\\[-1pt]
&=& 1 + \frac{u-u^2}{m+u} \leq1 + \frac{1}{2}\epsilon.
\end{eqnarray*}
\upqed
\end{pf}

\begin{pf*}{Proof of Theorem~\ref{2d-largetheta-upper-thm}}
We can divide the $d$-dimensional Hamming torus into $n^{d-2}$ disjoint
$2$-dimensional planes all parallel to the $e_1,e_2$-plane. Our goal is
to show that at least one of these planes are internally spanned with
high probability when $p=n^{-\alpha}$ with $\alpha= 1+2/\theta+
\sqrt{8(d-2.1)}/\theta^{3/2}$. The number of these 2-planes that are
internally spanned is binomially distributed, so we need only to show
that the expected number of internally spanned planes tends to
infinity. The expected number of internally spanned planes is
\[
n^{d-2} \sigma_\theta \bigl(2,n^{-\alpha} \bigr) \geq C
n^{d-2 - \beta(\alpha)}
\]
by Lemma~\ref{2d-span-lb-lem}. By applying Lemma~\ref{abound}, we see
that when $\theta=2k-1$ is odd
\begin{eqnarray*}
\beta(\alpha) &=& \alpha k^2 - (k+1)^2+a(a+1) - \alpha
a(a-1)
\\[-1pt]
&\leq&\alpha k^2 - (k+1)^2+ \frac{1}{\alpha- 1} + 1 +
\frac
{1}{2}(\alpha-1)
\\[-1pt]
&=& \biggl(1 + \frac{2}{\threshold}+\frac{\sqrt {8(d-2.1)}}{\threshold^{3/2}} \biggr) \biggl(
\frac{\threshold
+1}{2} \biggr)^2 - \biggl(\frac{\threshold+3}{2}
\biggr)^2 + \frac
{\threshold}{2 + \sqrt{8(d-2.1)/\threshold}}
\\[-1pt]
&&{} + 1+ \frac
{1}{\threshold}+ \frac{\sqrt{8(d-2.1)}}{2\threshold^{3/2}}
\\[-1pt]
&\leq&-\frac{\threshold}{2}+\frac{3}{2\threshold}+\frac{\sqrt {8(d-2.1)}}{4} \bigl(
\threshold^{1/2} + 2\threshold^{-1/2} + \threshold
^{-3/2} \bigr)
\\[-1pt]
&&{} + \frac{\threshold}{2} \biggl(1- \frac{\sqrt {8(d-2.1)}}{2\threshold^{1/2}}+\frac{8(d-2.1)}{4\threshold}
\biggr) + \frac{\sqrt{8(d-2.1)}}{2\threshold^{3/2}}
\\[-1pt]
& =& d-2.1 + \frac{3}{2\threshold}+\frac{\sqrt{8(d-2.1)}}{4} \bigl(2\threshold^{-1/2}
+ 3\threshold^{-3/2} \bigr)
\\[-1pt]
& <& d-2,
\end{eqnarray*}
where the last inequality holds for $\theta$ large relative to $d$,
and in the fourth line we used the inequality $(1+x)^{-1} \leq1-
x+x^2$ for $x>0$. This implies that the expected number of internally
spanned 2-dimensional planes tends to infinity with $n$, and completes
the proof for odd $\theta$. The proof for even $\theta$ is analogous.
\end{pf*}

The next theorem is a simple but powerful observation, which we refer
to as the dimension reduction inequality.

%
\begin{theorem}
\label{dimred-thm}
For any $d\geq2$, $\threshold\geq2$, and $1\leq d' \leq d-1$
%
\begin{equation}
\label{dim-reduction} \sigma_{\threshold}(d,p) \geq\sigma _{\threshold}
\bigl(d-d',\sigma_{\threshold} \bigl(d',p \bigr)
\bigr).
\end{equation}
\end{theorem}

\begin{pf}
We can subdivide the $d$-dimensional Hamming torus into $n^{d-d'}$
disjoint sub-Hamming tori of dimension $d'$. The probability of
internally spanning a fixed sub-Hamming torus is $\sigma_{\threshold
}(d',p)$, and the initially open sets in the \mbox{sub-Hamming} tori are
mutually independent. Therefore, we may identify each $d'$-dimensional
sub-Hamming torus with a single vertex, which is open independently
with probability $\sigma_\threshold(d',p)$, and the result is a
random subset of a $(d-d')$-dimensional Hamming torus that spans with
probability $\sigma_\threshold(d-d',\sigma_\threshold(d',p))$. If
this procedure spans the $(d-d')$-dimensional Hamming torus, then the
original configuration in the $d$-dimensional graph will span as well.
\end{pf}

Since we can compute bounds for $\sigma_{\threshold}(2,p)$ and
$\sigma_{\threshold}(1,p)$ for all $\threshold$ and $p$, the
dimension reduction inequality yields lower bounds on the critical
exponents for all $d$ and $\threshold$. In some cases, the lower
bounds obtained this way match our upper bounds, so we can precisely
compute the critical exponent. For instance, when $d = 3$ and
$\threshold=4$ we see that the critical exponent is $\alpha_{\mathrm{c}} =
1+d/\threshold= 7/4$. In this case, if $\alpha= (7-\epsilon)/4$
with $0<\epsilon<1$ then Lemma~\ref{2d-span-lb-lem} with $k=2$
implies that $\sigma_{4}(2,n^{-\alpha}) \geq cn^{6 - 4\alpha} =
cn^{-1+\epsilon}$. Then, since $\sigma_\threshold(d,p)$ is
increasing in $p$,
\[
\sigma_4 \bigl(3,n^{-\alpha} \bigr) \geq\sigma_4
\bigl(1,\sigma_4 \bigl(2,n^{-\alpha} \bigr) \bigr) \geq
\sigma_4 \bigl(1, c n^{-1+\epsilon} \bigr) = P \bigl( \Bin \bigl(n,
cn^{-1+\epsilon} \bigr) \geq4 \bigr) \to1.
\]
Theorem~\ref{molar} implies that $1+d/\threshold$ is always an upper
bound for the critical exponent, so in the case $d=3$, $\threshold=4$
the critical exponent is $7/4$.

%
%
\begin{figure}[b]

\includegraphics{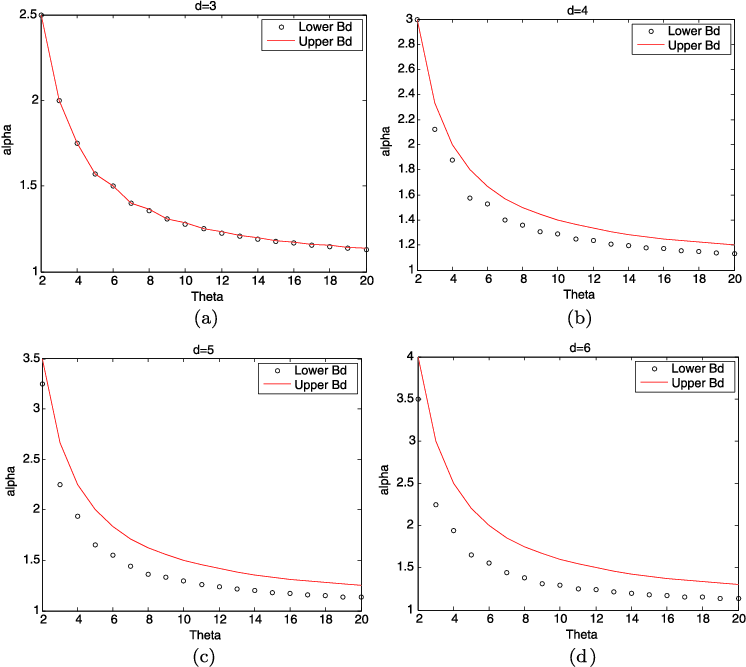}

\caption{Upper and lower bounds for the critical exponent when $p
\asymp n^{-\alpha}$.}
\label{alpha-bounds}
\end{figure}

As a second example of how to apply Lemma~\ref{2d-span-lb-lem} and
Theorem~\ref{dimred-thm}, consider the case $d=6$, $\threshold=5$.
Applying dimension reduction and Lemma~\ref{2d-span-lb-lem} twice yields
\[
\sigma_5 \bigl( 6,n^{-\alpha} \bigr) \geq\sigma_5
\bigl(4,\sigma_5 \bigl(2,n^{-\alpha} \bigr) \bigr) \geq
\sigma_5 \bigl(4,Cn^{-\beta(\alpha)} \bigr) \geq\sigma_5
\bigl(2,c n^{-\beta
(\beta(\alpha))} \bigr).
\]
The last term above tends to $1$ as $n\to\infty$ if $\beta(\beta
(\alpha)) < 4/3$ by Theorem~\ref{2d-thm}, so finding the supremum
over $\alpha$ satisfying this inequality gives a lower bound on the
critical exponent in this case. With a little help from Matlab, we can
numerically compute this supremum, and generate lower bounds for other
$d$ and $\threshold$. See Figure~\ref{alpha-bounds} for plots of
upper and lower bounds on $\alpha_{\mathrm{c}}$ for $d \in\{2,3,4,5,6\}$ and
$\theta\in\{2, \ldots, 20\}$. Table~\ref{exact-exponents} lists all
cases for which our upper and lower bounds match when $d=3$, and a few
cases for which they conspicuously do not ($\threshold= 8,10,12$). The
upper bounds in the table are the smaller of $1+3/\threshold$ and the
bounds from Theorem~\ref{d3upperbound}---either $1+8/(3\threshold
-1)$ or $1+8/(3\threshold-2)$, depending on whether $\threshold$ is
odd or even.

%
%
\begin{table}
\tablewidth=\textwidth
\tabcolsep=0pt
\caption{Upper and lower bounds for the critical exponent when $d=3$}
\label{exact-exponents}
\begin{tabular*}{\textwidth}{@{\extracolsep{\fill}}lccccccccccc@{}}
\hline
&\multicolumn{11}{c}{$\boldsymbol{\threshold}$}\\[-5pt]
&\multicolumn{11}{c}{\hrulefill}\\
\textbf{Bound} & \textbf{2} & \textbf{3} & \textbf{4} &\textbf{5} & \textbf{6}& \textbf{7}
& \textbf{8}& \textbf{9}& \textbf{10}& \textbf{11}&\textbf{12} \\
\hline
Lower  & $5/2$ & 2 & $7/4$ & $11/7$& $3/2$ & $7/5$ & $19/14$ &
$17/13$ & $23/18$ &$5/4$ & $27/22$ \\
Upper & $5/2$ & 2 & $7/4$ & $11/7$& $3/2$ & $7/5$ &15/11 & 17/13&
9/7 & 5/4 & 21/17\\
\hline
\end{tabular*}
\tabnotetext[]{}{\textit{Note}: If $p\asymp n^{-\alpha}$ and $\alpha$ is larger than the upper bound,
then spanning will not occur with high probability, while if $\alpha$
is smaller than the lower bound then spanning will occur with high probability.}
\end{table}
%
%
%
\begin{figure}[b]

\includegraphics{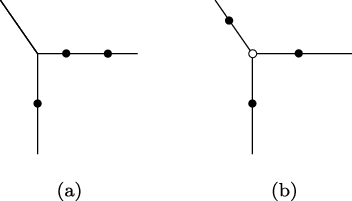}

\caption{Without one of these configurations appearing somewhere in
the graph at time 0, nothing will become open at time 1 when
$d=\theta=3$. The open circle in \textup{(b)} is to emphasize that this ``Basic''
configuration is with respect to a focal vertex which will become open
at time 1. The ``Line'' configuration in \textup{(a)} is indexed with respect to
the line which contains two open points, and the single open vertex off
of the horizontal line signifies that at least one vertex on one of the
two planes containing the focal line must be open.}
\label{tospan0}
\end{figure}

\section{A precise three-dimensional result}
\label{3d-precise}


In this section, we precisely compute the limiting spanning probability
in the case $d=3$ and $\threshold=3$. As computed in Section~\ref{suffforplanes}, the critical exponent in this case is $\alpha=2$ (see
Table~\ref{exact-exponents}), so we consider the scaling $p = a n^{-2}$ when $a>0$ is a constant.
\comment{
%
%
\begin{theorem}
\label{3d-spanning-thm}
Let $d=3$, $\threshold=3$ and $p = an^{-2}$ with $a>0$. Then as $n\to
\infty$
%
\begin{equation}
\prob_{p}(\omega_\infty\equiv1) \to1 - e^{-a^3 -
(3/2)a^2(1-e^{-2a})}
\biggl[\frac{3}{2} a^2 \bigl( \bigl(e^{-a}+ae^{-3a}
\bigr)^2-e^{-2a} \bigr)e^{-a^2e^{-2a}} + e^{a^3e^{-3a}}
\biggr].
\end{equation}
\end{theorem}
}

The resulting limit in Theorem~\ref{3d-spanning-thm} is a simplified
expression for a probability involving Poisson random variables with
means depending on $a$. Indeed, to compute the spanning probability, we
identify the minimal ingredients that lead to spanning, and show that
their frequencies of occurrence in $\omega_0$ converge jointly to
independent Poisson random variables by using the Chen--Stein method
\cite{poissonbook}. First, we identify two fundamental configurations,
which we will define carefully later: points that see at least one open
vertex in each direction [Figure~\ref{tospan0}(b)] and lines that
contain at least two open vertices and at least one more open vertex in
the same plane [Figure~\ref{tospan0}(a)]. At least one of these
configurations is necessary (in the limit) for spanning because lines
that contain 3 or more open vertices do not appear when $p= an^{-2}$,
as the expected number of such lines is $O(n^2 (np)^3) = O(n^{-1})$.
Note that in the definitions of our configurations we allow for there
to be three or more open vertices in a line, even though this is
unlikely to occur for large~$n$. This is to maintain some monotonicity
of the events, and simplifies the Poisson convergence proofs. Each
fundamental configuration also has a corresponding ``enhanced''
configuration (Figures~\ref{tospan2} and \ref{tospan5}), which
requires additional open vertices in certain planes. Each of these
configurations has nonzero probability in the limit, and affects the
limiting spanning probability.
%
%
\begin{figure}[t]

\includegraphics{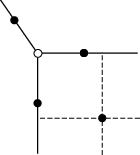}

\caption{``Enhanced Basic'': First the two lines containing the open
circle in the front plane will be spanned, followed by the two dotted
lines then the front plane. Once a plane is spanned, the rest of the
graph is likely to be spanned (see the last paragraph in the proof of
Lemma \protect\ref{good-span-lem}).}
\label{tospan2}
\end{figure}

We must now determine which combinations of these ingredients are
asymptotically necessary and sufficient for spanning. This is
summarized as follows:
\begin{longlist}[(4)]
\item[(1)] At least one ``basic'' configuration like that in Figure~\ref{tospan0}(b), AND at least one ``line'' configuration like that in
Figure~\ref{tospan0}(a); OR

\item[(2)] At least one ``enhanced basic'' configuration like that in
Figure~\ref{tospan2}; OR

\item[(3)] At least one ``line'' configuration, AND at least one askew
(nonparallel, nonintersecting) line that contains at least two open
vertices (see the configuration in Figure~\ref{tospan3}); OR

\item[(4)] At least two ``line'' configurations like the one in Figure~\ref{tospan0}(a); OR

\item[(5)] At least one ``enhanced line'' configuration like those in
Figure~\ref{tospan5}.
\end{longlist}
We call $\omega_0$ \emph{good} if it contains at least one of the
recipes (1)--(4) described above; a formal definition is given below.
The event $\{\omega_0$ is good\} is asymptotically equivalent to the
event \{$\omega_0$ spans\} in the sense of the following lemma.

%
%
\begin{figure}

\includegraphics{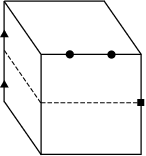}

\caption{This configuration leads to the front plane being spanned,
and the graph is likely to be spanned. There is a ``line'' configuration
with respect to the line that contains the two closed circles---the
rectangle in the front plane completes the configuration and leads to
the spanning of the top line in two steps. After the line with two
circles is spanned, the line with two triangles is now in a ``line''
configuration, and is spanned in two more steps. The vertex at the
intersection of the dotted line and the line with the triangles is now
open, and leads to the vertex at the intersection of the dotted lines
becoming open, which leads to the spanning of the front plane in three
more steps. Note that it is crucial for the lines with the circles and
triangles to be askew---if these lines were parallel then the front
plane would not be spanned without additional help.}
\label{tospan3}
\end{figure}
%
%
\begin{figure}[b]

\includegraphics{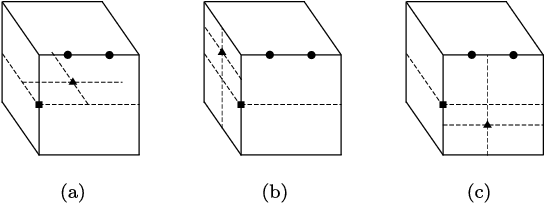}

\caption{``Enhanced Line'': These configurations
labeled by \textup{(a)}, \textup{(b)} and \textup{(c)} (and any rotations or
shifts of them) are likely to span. The triangle vertex will cause a
second line in the front plane to be spanned, thus the full front plane
will be spanned if there is an additional open vertex anywhere in the
graph that is not coplanar with this line or the line with two circles.
Once a plane is spanned, the rest of the graph is likely to be spanned.}
\label{tospan5}
\end{figure}

%
\begin{lemma}
\label{good-span-lem}
If $d=\theta=3$ and $p=an^{-2}$, then as $n\to\infty$
\[
\prob(\omega_0 \mbox{ is } good) - \prob(\omega_\infty
\equiv1) \to0.
\]
\end{lemma}

To formally define the event $\{\omega_0 \mbox{ is } \good\}$, and
for the proofs that follow, we need to introduce some notation.

\subsection*{Notation}

Let $\vect{e_1}, \vect{e_2}, \vect{e_3}$ denote the standard basis
vectors in $\mathbb{R}^3$. For $\vect{v},\vect{w}\in V$ let $d(\vect
{v},\vect{w})$ be the number of nonzero coordinates of $\vect
{v}-\vect{w}$. Let $\mathcal{N}(\vect{v}) = \{\vect{w} \in V\dvtx
d(\vect{v},\vect{w}) = 1\}$ denote the neighborhood of $\vect{v}$,
and for $A\subseteq V$ let $\mathcal{N}(A) = \bigcup_{\vect{v}\in
A}\mathcal{N}(\vect{v}) \setminus A$.

The basic and enhanced basic configurations will be indexed by
vertices, while the line and enhanced line configurations will be
indexed by lines. So, we let
\[
\lineset= \bigl\{\Ln\subseteq V\dvtx \llvert\Ln\rrvert= n \mbox{ and } \forall
\vect{v},\vect{w}\in\Ln, d(\vect{v}, \vect{w}) \leq1 \bigr\}
\]
be the set of lines in $V$. Also, for $i=1,2,3$, let
\[
\lineset_i = \bigl\{\Ln\in\lineset\dvtx \forall\vect{u},\vect{v}
\in \Ln, \exists m = m(\vect{u},\vect{v}) \in\mathbb{Z}\mbox{ s.t. }\vect{u} =
\vect{v} + m \vect{e_i} \bigr\}
\]
denote the collection of lines in $V$ parallel to the coordinate axis
in the $\vect{e_i}$ direction. For the duration of this paper, we will
use $\Ln$ to refer to a generic line.

In order to apply the Chen--Stein method, we let $\basic$, $\lin$, $\oline$,\linebreak[4]
$\ebasic$, $\eline$ and $\neline$ be the random variables that count
the number of occurrences of the corresponding configurations in
$\omega_0$, which we now define carefully. The relevant events are a
bit difficult to describe, so we refer the reader to Figures~\ref{tospan0}--\ref{tospan5} for guidance.

Define the \emph{basic} event, for $\vect{v} \in V$, to be
\begin{eqnarray*}
\event^{\mathrm{B}}_{\vect{v}} &=& \bigl\{ \exists\vect{w_1},
\vect{w_2}, \vect{w_3} \in\omega_0
\setminus\{\vect{v}\} \mbox{ and }\exists m_1, m_2,
m_3 \in\mathbb{Z}
\\
&&\hspace*{60pt}\mbox{s.t. }\vect{v} = \vect{w_i} +
m_i \vect{e_i} \mbox{ for } i=1,2,3 \bigr\}.
\end{eqnarray*}
As Figure~\ref{tospan0}(b) indicates, the basic event occurs at $\vect
{v}$ if $\vect{v}$ has at least one initially open neighbor in each
basis direction. Define the \emph{enhanced basic} event, for $\vect
{v} \in V$, to be
\begin{eqnarray*}
\event^{\mathrm{EB}}_{\vect{v}} &=& \bigl\{\exists\vect{w} \in \omega
_0\mbox{ s.t. }d(\vect{v}, \vect{w}) = 2, \mbox{ and } \exists\vect
{w_1}, \vect{w_2}, \vect{w_3} \in
\omega_0 \setminus \bigl(\mathcal{N}(w)\cup\{\vect{v}\} \bigr)
\\
&&\hspace*{48pt} \mbox{and }\exists m_1, m_2,
m_3 \in \mathbb{Z}\mbox{ s.t. }\vect{v} = \vect{w_i} +
m_i \vect{e_i} \mbox{ for } i=1,2,3 \bigr\}.
\end{eqnarray*}
As Figure~\ref{tospan2} indicates, the enhanced basic event occurs at
$\vect{v}$ if the basic event occurs at $\vect{v}$ and there is at
least one open vertex in one of the planes containing $\vect{v}$ that
is not a neighbor of $\vect{v}$. Further, this additional open vertex
should not be collinear with the sole open neighbor of $\vect{v}$ in
any direction; if there were two open neighbors of $\vect{v}$ in a
single direction, then we could allow the additional open vertex to be
collinear with one of them, but this event is rare. Let $I^{\mathrm
{B}}_{\vect
{v}}$ be the indicator random variable for the event $\event^{\mathrm{B}}_{\vect
{v}}$, so $\basic= \sum_\vect{v} I^{\mathrm{B}}_{\vect{v}}$, and let
$I^{\mathrm{EB}}
_{\vect{v}}$ be the indicator random variable for the event $\event
^{\mathrm{EB}}_\vect{v}$, so $\ebasic= \sum_{\vect{v}} I^{\mathrm
{\mathrm{EB}}}_{\vect{v}}$. In
general, we will denote by $I_\dagger^*$ the indicator of the event
$\event_\dagger^*$.

For each line $\Ln\in\lineset$, we define the \emph{line} event
\begin{eqnarray*}
\event^{\mathrm{L}}_\Ln&=& \bigl\{ \llvert\Ln\cap
\omega_0\rrvert=2, \bigl\llvert\mathcal{N}(\Ln)\cap
\omega_0\setminus\mathcal{N}(\Ln\cap\omega_0) \bigr
\rrvert\geq1 \bigr\}
\\
&&{}\cup \bigl\{ \llvert\Ln\cap\omega_0\rrvert \geq3, \bigl\llvert
\mathcal{N}(\Ln)\cap\omega_0 \bigr\rrvert \geq1 \bigr\}.
\end{eqnarray*}
As Figure~\ref{tospan0}(a) suggests, the line event occurs at $\Ln$ if
$\Ln$ contains at least two initially open vertices, and there is at
least one additional open vertex in the same plane as~$\Ln$. This
additional open vertex should not be in the neighborhood of the two
open vertices in $\Ln$, though if there are three or more open
vertices in $\Ln$ then the location of the additional vertex does not
matter. We now define $\lin= \sum_{\Ln\in\lineset} I^{\mathrm
{L}}_\Ln$,
and because we will also need to count the number of line events in a
particular direction [for case (3) in the recipe for spanning], for
$i=1,2,3$ we let $\lin_i = \sum_{\Ln\in\lineset_i} I^{\mathrm
{L}}_\Ln$.
For each $\Ln\in\lineset$, we define the \emph{$\emptyset$-line} event
\begin{eqnarray*}
\event^{\emptyset \mathrm{L}}_\Ln&= \bigl\{ \llvert\Ln\cap
\omega_0\rrvert\geq2 \bigr\} \setminus\event^{\mathrm{L}}_\Ln,
\end{eqnarray*}
and let $I^{\emptyset\mathrm{L}}_\Ln$ be the corresponding indicator
random variable so\vspace*{-1pt}
$\oline= \sum_{\Ln\in\lineset} I^{\emptyset\mathrm{L}}_\Ln$ and
for $i=1,2,3$,
$\oline_i = \sum_{\Ln\in\lineset_i} I^{\emptyset\mathrm{L}}_\Ln
$. The $\emptyset
$-line event occurs at $\Ln$ if $\Ln$ contains at least two initially
open vertices, and there are no other open vertices in the same plane
as $\Ln$ (except possibly those that are collinear with one of the two
open vertices in $\Ln$).

For each line $\Ln\in\lineset$, we define the \emph{enhanced line} event
\begin{eqnarray*}
\event^{\mathrm{EL}}_\Ln&=& \bigl\{ \llvert\Ln\cap
\omega_0\rrvert=2 \mbox{ and } \exists\vect{v}\in\mathcal{N}(\Ln)
\cap\omega_0\setminus\mathcal{N}(\Ln\cap\omega_0)
\\
&&\hspace*{29pt}\mbox{s.t. } \bigl\llvert\mathcal{N} \bigl(\mathcal{N}(\vect{v})
\bigr) \cap \omega_0 \setminus\mathcal{N} \bigl(\Ln\cap
\mathcal{N}(v) \bigr) \bigr\rrvert\geq1 \bigr\}
\\
&&{} \cup \bigl\{ \llvert\Ln\cap\omega_0\rrvert\geq3, \exists\vect
{v}\in\mathcal{N}(\Ln)\cap\omega_0\mbox{ s.t. } \bigl\llvert
\mathcal{N} \bigl(\mathcal{N}(\vect{v}) \bigr) \cap\omega_0
\setminus\mathcal{N} \bigl(\Ln\cap\mathcal{N}(v) \bigr) \bigr\rrvert\geq1 \bigr\}
\end{eqnarray*}
and let $I^{\mathrm{EL}}_\Ln$ be the corresponding indicator random
variable\vspace*{1pt} so\break
$\eline= \sum_{\Ln\in\lineset} I^{\mathrm{EL}}_\Ln$ and for $i=1,2,3$,
$\eline_i = \sum_{\Ln\in\lineset_i} I^{\mathrm{EL}}_\Ln$. For
the enhanced
line event to occur at $\Ln$, a line configuration must appear in
$\omega_0$ at $\Ln$ and there must be at least one additional open
vertex. This additional open vertex is coplanar with the open vertex in
$\mathcal{N}(\Ln)$ from the line configuration (there may be more
than one), but is not counted if it is collinear with this vertex or on
the other plane containing~$\Ln$. Finally, define the nonenhanced line event
\[
\event^{\mathrm{NEL}}_\Ln= \event^{\mathrm{L}}_\Ln
\setminus\event^{\mathrm{EL}}_\Ln
\]
and its corresponding indicator $I^{\mathrm{NEL}}_\Ln$, so that
$I^{\mathrm{NEL}}_\Ln=
I^{\mathrm{L}}_\Ln- I^{\mathrm{EL}}_\Ln$ for every\break $\Ln\in
\lineset$, $\neline= \lin
- \eline$ and for $i=1,2,3$,\break $\neline_i = \lin_i - \eline_i$.

Now we define the event that $\omega_0$ is good by
\begin{eqnarray*}
\{\omega_0 \mbox{ is } \good\} &=& \{\basic\geq1, \lin\geq1\} \cup
\{ \ebasic\geq1 \}
\\
&&{}\cup\bigcup_{i=1}^3 \biggl\{
\lin_i \geq1, \sum_{j\neq i}
\oline_j \geq1 \biggr\}
\\
&&{} \cup\{\lin\geq2\} \cup\{\eline\geq1\}.
\end{eqnarray*}
The third term above covers the scenario in Figure~\ref{tospan3} when
$\lin\leq1$, which is otherwise covered by the event $\{\lin\geq2\}
$. Using inclusion--exclusion, exploiting obvious symmetries of the
graph, and combining like terms:
%
\begin{eqnarray}
\label{goodevent} %
&&\prob( \omega_0 \mbox{ is good})\nonumber \\
&&\qquad = \prob(\basic\geq1, \lin=1) + \prob({ \ebasic\geq1, \lin=0})
\nonumber
\\
&&\quad \qquad {} + \prob(\lin\geq2)\nonumber\\
&&\quad \qquad {} + \prob(\basic=0,\eline=1, \neline=0)
\\
&&\quad \qquad {} + 3 \prob(\basic=0,\neline_1=1,\nonumber\\
&&\hphantom{\quad \qquad {} + 3 \prob(}\neline_2+
\neline_3=0,
\nonumber
\\
& &\hspace*{95pt}\eline=0,\oline_2+\oline_3\geq1).
\nonumber
\end{eqnarray}
Therefore, once we compute the probabilities in (\ref{goodevent}),
Lemma~\ref{good-span-lem} implies Theorem~\ref{3d-spanning-thm}.
Lemma~\ref{3d-poisson-lem} allows us to do just this, and is followed
by the proof of Lemma~\ref{good-span-lem}. The proof of Lemma~\ref
{3d-poisson-lem} uses the Chen--Stein method, and is outlined in
the \hyperref[ap:poisson]{Appendix}.

%
\begin{lemma}
\label{3d-poisson-lem}
If $p = an^{-2}$, then as $n\to\infty$  Table  \ref{tab2}  gives the
means of the random variables appearing in (\ref{goodevent}).
\begin{table}
\tablewidth=\textwidth
\tabcolsep=0pt
\caption{Means of the random variables appearing in (\protect\ref{goodevent})}\label{tab2}
\begin{tabular*}{\textwidth}{@{\extracolsep{\fill}}lc@{}}
\hline
\textbf{Random variable} & \textbf{Mean} \\
\hline
$\basic$ & $a^3$ \\
$\ebasic$ & $a^3 (1 - e^{-3a})$ \\
$\lin$ & $\frac{3}{2} a^2 (1 - e^{-2a})$ \\
$\oline_i$ & $\frac{1}{2} a^2 e^{-2a}$ \\
$\neline_i$ & $\frac{1}{2}a^2 [(e^{-a}+ae^{-3a} )^2 - e^{-2a} ]$ \\
$\eline$ & $\frac{3}{2}a^2 [1 -
(e^{-a}+ae^{-3a} )^2 ]$ \\ \hline
\end{tabular*}
\end{table}
Furthermore, the two random variables $\ebasic$ and $\lin$ converge
jointly in distribution to independent Poisson random variables with
the above means, as do the eight random variables $\basic$, $\eline$,
and for $i=1,2,3$, $\neline_i$ and $\oline_i$.
\end{lemma}

%
%
\begin{remark}
\label{3d-poisson-rem}
Lemma~\ref{3d-poisson-lem} allows us to compute the limiting
probability in (\ref{goodevent}) by treating all of the random
variables that appear as independent Poisson random variables with the
means given by the table. The means that appear in the limit are
straightforward to compute. For example, to compute the expected number
of basic events, the probability that a fixed vertex has at least one
initially open neighbor in each direction is $\sim(np)^3 = a^3 / n^3$,
and there are $n^3$ vertices at which a basic configuration can be
centered. To obtain the expected number of enhanced basic
configurations, observe that a fixed vertex must first see a basic
configuration, then independently at least one of the $3(n-2)^2$
coplanar but not collinear vertices must be present. This has
probability $1 - (1-p)^{3(n-2)^2} \sim1 - e^{-3a}$ of occurring.
\end{remark}

\begin{pf*}{Proof of Lemma~\ref{good-span-lem}}
We will first show that spanning does not occur with high probability
when $\omega_0$ is not good. The expected number of lines that
contain at least three initially open vertices
is $\sim3n^2 { n \choose 3}p^3=O(n^{-1})$, so at least one line
configuration or basic configuration is necessary for any vertices to
become open after one step.

Any vertex that becomes open in the second step must be neighbors with
at least one vertex that becomes open in the first step, that is, with
a vertex in $\omega_1\setminus\omega_0$. If $\lin= 0$ and $\ebasic
=0$ then any two basic events located at vertices $\vect{v}$ and
$\vect{w}$ cannot be coplanar unless $\mathcal{N}(\vect{v})\cap
\mathcal{N}(\vect{w}) \subseteq\omega_0$, otherwise a line or an
enhanced basic configuration would exist. The probability that there
exist two vertices, $\vect{v}$ and $\vect{w}$, with $I^{\mathrm
{B}}_{\vect
{v}}I^{\mathrm{B}}_{\vect{w}} = 1$, $d(\vect{v},\vect{w})=2$ and
$\mathcal
{N}(\vect{v})\cap\mathcal{N}(\vect{w}) \subseteq\omega_0$\vspace*{-1pt} is at
most $3n {n^2 \choose2} (np)^2 p^2 = O(n^{-1})$, so with high
probability there are no coplanar basic events. Therefore, no pair of
vertices in $\omega_1 \setminus\omega_0$ have a common neighbor, and
no vertex in $\mathcal{N}(\omega_1 \setminus\omega_0) \setminus
\omega_0$ has more than one neighbor in $\omega_0$ (or else a line or
enhanced basic configuration would have existed in $\omega_0$). This
implies that no vertices can become open in the second step, so
spanning cannot occur with high probability when $\lin= 0$ and
$\ebasic= 0$.

Also, if simultaneously $\neline_1 = 1$, $\neline_2+\linebreak[4] \neline_3=0$, $\basic=
0$, $\eline= 0$ and $\oline_2 +\linebreak[4]  \oline_3 = 0$ then spanning is
unlikely to occur. The sole line configuration will span the focal
line, $\Ln$, after two steps. There may be parallel lines that contain
two occupied vertices, but they cannot be coplanar with $\Ln$ or else
the line configuration would be enhanced. These parallel lines will not
span the cube as their neighborhoods do not intersect $\Ln$, so no
other vertices will become open after two steps. Therefore, $\prob(\{
\omega_\infty\equiv1\} \setminus\{\omega_0 \mbox{ is } \good\})
\to0$.

The probability of $\omega_0$ containing a basic configuration and a
line configuration that share a plane [i.e., there exist $\vect{v}$
and $\Ln$ so that $I^{\mathrm{B}}_{\vect{v}} I^{\mathrm{L}}_\Ln=1$
and $\vect
{v} \in
\mathcal{N}(\Ln)\cup\Ln$] is at most $Cn^3 (n) (np)^3 (np)^2 =
O(n^{-1})$. Similarly, the probability of having two or more coplanar
line configurations is $O(n^{-1})$. Conditional on the complements of
these last two events, observe that a line configuration will cause a
basic configuration to become an enhanced basic configuration in two
steps. Likewise, a line configuration will cause a second line
configuration to become an enhanced line configuration in two steps;
and similarly a line configuration will with high probability cause an
askew line with two initially open vertices to become a line
configuration (and subsequently an enhanced line configuration).

Both the enhanced basic and enhanced line configurations lead to a
plane becoming open. Once a plane is open, two nonneighboring,
coplanar open vertices will cause another plane to become open, then
one more open vertex elsewhere will cause the rest of the graph to
become open. With probability exponentially close to 1, there are at
least $n^{1/2}$ planes with at least two nonneighboring open vertices
in $\omega_0$. Therefore, $\prob(\{\omega_0 \mbox{ is } \good\}
\setminus\{\omega_\infty\equiv1\})= O(n^{-1})$, and the two events
are asymptotically equivalent.
\end{pf*}

\section{Open one-dimensional subgraphs} \label{lines}
In this section,
we obtain an upper bound on the threshold probability for lines,
$p_{\mathrm{c}}(1,d)$. The main idea is the following. Assume that the line $\Ln$
contains $\noodle\le\threshold$ initially open vertices, that it
intersects one line with $\threshold-\noodle$ initially open sites
(not on $\Ln$), and that it intersects $\threshold$ other lines, each
with $\threshold- \noodle-1$ sites (not on $\Ln$) initially open.
Then after one step, $\Ln$ has $\noodle+1$ points open, and after two
steps, $\threshold$ points open. After three steps, $\Ln$ is
completely open. See Figure~\ref{fig:FL} for an illustration.

%
%
\begin{figure}[t]
%
%
%
%
%
%

\includegraphics{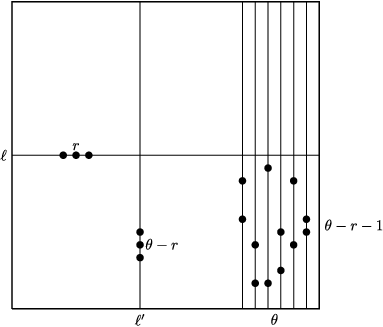}
\caption{An instance of the event $F_\Ln$. Here, $\threshold=6$,
$\noodle=3$. After one step, the intersection of lines $\Ln$ and $\Ln
'$ becomes open so $\Ln$ has $\noodle+1$ vertices open. At step 2,
the $\threshold$ intersections with $\Ln$ and the other $\threshold$
vertical lines become open. At step 3, all of $\Ln$ becomes open.}
\label{fig:FL}
\end{figure}

For a set $S \subseteq V$ and $x \in\N$, let $\Initial
(S, \geq x )$ be the event
that the set $S$ has at least $x$ points initially open, that is,
\[
\Initial(S, \geq x )= \biggl\{\sum_{v \in
S}
\omega_0(v) \geq x \biggr\}.
\]
For a point $v \in V$, let $P_{1,2}(v)$ be the $e_1,e_2$-parallel plane
through $v$:
\[
P_{1,2}(v)= \bigl\{(a_1,a_2,v_3,v_4,
\dots,v_d)\dvtx a_1,a_2 \in[n] \bigr\}.
\]
Let $\Ln_2(v)$ be the $e_2$-parallel line through $v$:
\[
\Ln_2(v) = \bigl\{(v_1,a_2,v_3,v_4,
\dots,v_d)\dvtx a_2 \in[n] \bigr\}.
\]
For any $e_1$-parallel line $\Ln$, define
\[
\Ln_l = \{w \in\Ln, w_1 < n/3\}, \qquad
\Ln_m = \{w \in\Ln, n/3 \leq w_1 \leq2n/3\},
\]
and
\[
\Ln_r = \{w \in\Ln, w_1 > 2n/3\}
\]
to be the left, middle and right thirds of $\Ln$. Define
\begin{eqnarray*}
\Gone{\Ln} &=& \biggl\{\sum_{v \in\Ln_m} \one_{\Initial
(\Ln_2(v), \geq\threshold-\noodle)}
\geq1 \biggr\},
\\
\Gtwo{\Ln} &=& \biggl\{\sum_{v \in\Ln_r} \one_{\Initial
(\Ln_2(v), \geq\threshold-\noodle-1 )}
\geq\threshold \biggr\}
\end{eqnarray*}
and
\[
F_\Ln= \Initial(\Ln_l, \geq\noodle)\cap \Gone{\Ln}\cap
\Gtwo{\Ln}.
\]
Notice that the event $F_\Ln$ depends only on the sites in
$P_{1,2}(v)$ for any $v \in\Ln$.
Also note that
\[
\Gone{\Ln}=\Gone{\Ln'}\quad \mbox{and}\quad \Gtwo{\Ln}=\Gtwo{
\Ln'}
\]
for any $e_1$-parallel lines $\Ln\neq\Ln'$ that lie in a common
$e_1,e_2$-parallel plane. Finally, note that $\Initial
(\Ln_l, \geq\noodle)$,
$\Gone{\Ln}$, and $ \Gtwo{\Ln}$ are independent, and $\Initial(\Ln_l,
\geq\noodle)$ and $\Initial(\Ln'_l, \geq\noodle)$ are independent.

We exhibit the role of $F_\Ln$ (see Figure~\ref{fig:FL}) in the
following lemma.

%
\begin{lemma} \label{twostep}
If $\Ln$ is a line parallel to the $e_1$ axis and $F_\Ln$ occurs,
then the entire line $\Ln$ is open after three steps.
\end{lemma}

%
\begin{remark} Computation of
$P(F_\Ln)$ is facilitated by independence of the three events.
A more natural definition would not restrict the orientations of the
lines, or
demand that the event happen in the left, middle or right sections
thereof, and would increase the probability by a constant factor,
independent of $n$.
\end{remark}

We set $\noodle= \lceil\frac{(d-1)\threshold}{d} \rceil
-1$ and $p = n^{-1-\sfrac{d}{\threshold}}f(n)$, where $f(n)$ is any
function such that $f(n) \to\infty$.
We will show that in this regime some line becomes open asymptotically
almost surely.
We will use the following elementary fact about the binomial distribution.

%
%
\begin{lemma} \label{binomlemma}
Assume that $S$ is $\Binomial(n,p)$, with large $n$ and $p=p(n)$, and
that $k$ does not depend on $n$.
If $np=O(1)$, then $P(S\ge k)\ge c (np)^k$ for some constant $c$
dependent on $k$.
If $np\to\infty$, then $P(S\ge k)\to1$.
\end{lemma}

%
%
\begin{lemma} \label{probA}
Fix $v \in V$ and $ \threshold, d \geq3$. Let $p=n^{-1-\sfrac
{d}{\threshold}}f(n)$ where $f(n) \to\infty$. Then for any $c>0$,
the probability that there exists an $e_1$-parallel line $\Ln$ in
$P_{1,2}(v)$ such that $F_\Ln$ occurs
is at least $cn^{2-d}$ for $n$ sufficiently large.
\end{lemma}
%

%
\begin{pf}
As the event in the statement is increasing, its probability is
monotone in $p$. Thus, we may
assume that $f(n)$ grows to $\infty$ as slowly as we need in the proof.

Note that when $\threshold, d \geq3$ then $\noodle d/\threshold\geq
1$ as
\[
\noodle d/\threshold\geq \biggl( \frac{(d-1)\threshold}{d} -1 \biggr) \frac{d}{\threshold}
= d-1- d/\threshold.
\]
The right-hand side is strictly
greater than 1 except if $d=\threshold=3$. We assume that at least one
of $d$ and $\threshold$
is at least 4, and leave the exceptional case to the reader.

The three events that define $F_\Ln$ depend on disjoint sets of sites,
so they are independent
and we compute their probabilities separately.
Furthermore, for the set of lines $\Ln$ we consider, the events
$\Gone{\Ln}$ and $\Gtwo{\Ln}$ do not depend on $\Ln$, which will
thus be dropped from
the notation.
For any $\Ln$, by Lemma~\ref{binomlemma}
\begin{eqnarray*}
\prob \bigl(\Initial(\Ln_l, \geq\noodle) \bigr) &\geq&
c_1 (np)^{\noodle}
\\
&\geq& c_1 \bigl(f(n)n^{-d/\threshold} \bigr)^{\noodle}.
\end{eqnarray*}
As this is $o(1/n)$, we can use Lemma~\ref{binomlemma} again to get that
\[
\prob \bigl(\exists\Ln\mbox{ such that }\Initial(\Ln_l, \geq
\noodle) \mbox{ occurs} \bigr) \geq c_2 n \bigl(f(n)n^{-d/\threshold}
\bigr)^{\noodle}.
\]
%
%
\comment{
\begin{eqnarray*}
\prob \bigl(\Initial(\Ln_l, \geq\noodle) \bigr) &\geq& {n/3
\choose\noodle}p^{\noodle}(1-p)^{n/3-\noodle}
\\
&\geq& c_1 (np)^{\noodle}
\\
&\geq& c_1 \bigl(f(n)n^{-d/\threshold} \bigr)^{\noodle}.
\end{eqnarray*}
Provided that $f(n)$ does not grow too quickly the right hand side is
much smaller than $1/n$ so we get
\[
\prob \bigl(\exists\Ln\mbox{ such that }\Initial(\Ln_l, \geq
\noodle) \mbox{ occurs} \bigr) \geq.5 c_1 n \bigl(f(n)n^{-d/\threshold}
\bigr)^{\noodle}.
\]
}

To estimate the second probability, observe that
\[
\prob \bigl(\Initial \bigl(\Ln_2(v), \geq\threshold-\noodle \bigr)
\bigr)\ge c_3 (np)^{\theta-r},
\]
which is $o(1/n)$, as $r<(d-1)\threshold/d$. Thus,
\begin{eqnarray*}
\prob(\Gonez) &\geq& c_4 n (np)^{\threshold- \noodle}
\\
&\geq& c_4 n \bigl(f(n)n^{-d/\threshold} \bigr)^{\threshold-
\noodle}.
\end{eqnarray*}

For the third probability,
\[
\prob \bigl(\Initial \bigl(\Ln_2(v), \geq\threshold-\noodle \bigr)
\bigr)\ge c_5 (np)^{\theta-r-1},
\]
and
\begin{eqnarray*}
n\cdot(np)^{\theta-r-1} &\geq& f(n)^{\threshold- \noodle-
1}n^{1-d+(\noodle+1)\sfrac
{d}{\threshold}} \to \infty
\end{eqnarray*}
as $n\to\infty$,
so Lemma~\ref{binomlemma} implies that
\[
\prob(\Gtwoz) \to1,
\]
and for large $n$ the probability is bounded below by a constant $c_6>0$.
Multiplying together the probabilities, we have that for any $c$ and
all sufficiently large $n$
\begin{eqnarray*}
&&\prob \bigl(\exists\Ln\mbox{ in }P_{1,2}(v)\mbox{ such that that
}F_\Ln\mbox{ occurs} \bigr)
\\
&&\qquad =\prob \bigl(\exists\Ln\mbox{ such that }\Initial(\Ln_l,
\geq \noodle) \bigr)\prob( \Gonez)\prob( \Gtwoz)
\\
&&\qquad \geq c_2n \bigl(f(n)n^{-d/\threshold} \bigr)^{\noodle}c_4n
\bigl(f(n)n^{-d/\threshold} \bigr)^{\threshold- \noodle}c_6
\\
&&\qquad = c_7f(n)^\threshold n^{2-d}
\\
&&\qquad > c n^{2-d},
\end{eqnarray*}
ending the proof.
\end{pf}

\begin{theorem} \label{bicuspid}
Suppose that $p = n^{-1-\sfrac{d}{\threshold}}f(n)$ with $f(n) \to\infty
$. Then\linebreak[4]  $\prob(\bigcup_\Ln F_\Ln)\to1$ as $n \to\infty$, where the
union is taken over all $e_1$-parallel lines. Thus, with probability
going to 1, some line becomes open after three steps.
\end{theorem}

\begin{pf}
We can choose $n^{d-2}$ distinct vertices $v_i$ such that
$P_{1,2}(v_i)$ are disjoint. Then the
events that there exist $\Ln$ in $P_{1,2}(v_i)$ where $F_\Ln$ occurs
are independent.
Moreover,
\[
n^{d-2}\prob \bigl(\exists\Ln\mbox{ in $P_{1,2}(v_i)$
such that $F_\Ln$ occurs} \bigr) \geq n^{d-2}cn^{2-d}=c
\]
for any fixed $c$.
Thus, $\prob(\bigcup_\Ln F_\Ln)\to1$ by Lemma~\ref{binomlemma}.
\end{pf}

%

%
\begin{theorem}\label{molar}
Assume that $p = n^{-1-d/\threshold}f(n)$, with $f(n) \to0$, then\break $\prob
(\mathrm{Above}\  \mathrm{Threshold}) \to0$.
\end{theorem}

\begin{pf}
Using the union bound,
\begin{eqnarray*}
\prob(\abovethresh) &\leq& \sum_{v \in V}\prob \biggl(
\sum_{w \sim v}\omega_0(w)\geq\threshold
\biggr)
\\
&=& n^d\prob \biggl(\sum_{w \sim v}
\omega_0(w)\geq\threshold \biggr)
\\
&\le& n^d{n \choose\threshold}p^\threshold
\\
&\le& f(n)^\threshold
\end{eqnarray*}
which approaches 0 as $n \to\infty$.
%
\end{pf}
%

\begin{pf*}{Proof of Theorem~\ref{la liga}}
Combining Theorems \ref{bicuspid} and \ref{molar} proves the result.
\end{pf*}


\section{Open two-dimensional subgraphs} \label{nessforplanes}
In previous sections, we have encountered several possibilities
for a vertex $v$ to become open:
\begin{itemize}
\item$v$ is initially open;
\item the neighborhood of $v$ has at least $\threshold$ vertices
initially open, causing $v$ to become open by time 1; and
\item a line containing $v$ has at least $\threshold(d-1)/d$ vertices
initially open,
with some additional open sites ``nearby'' (see Section~\ref{lines}).
\end{itemize}

Let $\D$ be the event that some plane eventually becomes open.
In this section, we show that if $p$ is sufficiently small then with
high probability all of the vertices that are eventually open satisfy a
condition like one of the three above.
By doing this, we prove an upper bound on the
probability of $\D$ and consequently a lower bound on the threshold
probability $p_{\mathrm{c}}(2,d)$.

Let $A$ be some integer, $1\leq A\leq\threshold$, which we will
specify later. Let $E$ be the event that there exists a vertex $v$ such that:
\renewcommand\thelonglist{(\arabic{longlist})}
\renewcommand\labellonglist{\thelonglist}
\begin{longlist}[(4)]
\item\label{aretha}$v$ is initially not open;
\item\label{thebigo} the neighborhood of $v$ has at most $A$ vertices initially open;

\item\label{wilsonpickett} each line containing $v$ has at most $A/2$ vertices initially
open; and
\item\label{samcooke}$v$ becomes open.
\end{longlist}
Our strategy to demonstrate that $\prob(\D)$ is small for
sufficiently small $p$ is to
show that $\prob(E)$ and $\prob(\D\setminus E)$ are both small.

For each vertex $v$, let $E_v$ be the event that $v$ satisfies
(1)--(4), and none among
such vertices becomes open earlier. If the event $E$ occurs, then there
must be
a first time a vertex satisfying (1)--(4) exists, thus $E \subseteq
\bigcup_v E_v$, and
consequently, $\prob(E) \le n^d \prob(E_v)$.

%
\begin{lemma}\label{cajunsprel1}
Suppose $p = o(n^{-1-\beta})$ with $\beta> (\frac{2d^2}{\threshold
-A}+1)\frac{2}{A}$. Fix a line $\Ln$.\vspace*{-1pt} The probability that $\Ln$
contains at least $\frac{\threshold-A}{2d}$ vertices $v$
that have at least $A/2$ initially open points in $\mathcal
{N}(v)\setminus\Ln$ is
\[
o \bigl(n^{\vafrac{\threshold-A}{2d}(1-\beta\sfrac{A}{2})} \bigr).
\]
\end{lemma}

\begin{pf}
The reduced neighborhoods $\mathcal{N}(v)\setminus\Ln$, $v\in\Ln$,
are pairwise
disjoint, and in each
the number of initially open vertices is a $\Binomial((d-1)(n-1),p)$
random variable.
The probability that such a random variable is at least $A/2$ is bounded
by a constant times $(np)^{A/2}=o (n^{-\beta A/2} )$. These
random variables are independent, thus the probability that at least
$\frac{\threshold-A}{2d}$ of them
are at least $A/2$ is $o ((n\cdot n^{-\beta A/2})^{\vafrac
{\threshold-A}{2d}} )$.
\end{pf}

%
%
\begin{lemma}\label{cajunsprel2}
Assume $p$ satisfies the same bound as in Lemma~\ref{cajunsprel1}.
Fix a line $\Ln$. The probability that $\Ln$ has at least $\frac
{\threshold-A}{2d}$ vertices $w$,
for which there exists a line $\Ln'\ne\Ln$ through $w$ such that
$\Ln'\setminus\{w\}$ contains at least $A/2$ initially open points is
\[
o \bigl(n^{\vafrac{\threshold-A}{2d}(1-\beta\sfrac{A}{2})} \bigr).
\]
\end{lemma}

\begin{pf} We need to bound the probability of at least
$\frac{\threshold-A}{2d}$ successes in
$n(d-1)$ independent trials, each of which is a success with
the probability that a given line has at least $A/2$ points initially open.
Same estimates as in the proof of Lemma~\ref{cajunsprel1} apply.
\end{pf}

%
%
\begin{lemma} \label{cajuns}
Assume $p$ satisfies the same bound as in Lemma~\ref{cajunsprel1}.
Then $\prob(E) \to0$ as $n \to\infty$.
\end{lemma}

\begin{pf}
As we have already observed, $\prob(E) \le n^d\prob(E_v)$. Now, if
$E_v$ occurs, by (2) at least $\threshold-A$ vertices in
the neighborhood of $v$ must be initially closed but become open
strictly before $v$; therefore, they violate at least one of (1)--(4).
But since they are not
open initially and become open, they must violate one of (2) or (3). By the pigeonhole principle, of
the $d$ lines through $v$, at least one must either contain $\frac
{\threshold-A}{2d}$ vertices which violate (2), or $\frac
{\threshold-A}{2d}$ vertices which violate (3).

By Lemmas \ref{cajunsprel1} and \ref{cajunsprel2}, each of these
happens with probability
\[
o \bigl( n^{\vafrac{\threshold-A}{2d}(1-\beta\sfrac{A}{2})} \bigr).
\]
Rearranging using the inequality $\beta> (\frac{2d^2}{\threshold
-A}+1)\frac{2}{A}$, we see that
$\prob(E_v)=o(n^{-d})$, as claimed.
\end{pf}

%
%
\begin{lemma} \label{activenotE}
Let $p = n^{-1-\beta}$, with $\beta> 0$, and assume $A\ge4$.
Then $\prob(\mathrm{Plane}\allowbreak  \mathrm{Active}\setminus E) \to0$ as $n \to\infty$.
\end{lemma}

\begin{pf}
There are ${d \choose2}n^{d-2}$ planes, $\plane$, and
$\D= \bigcup_\plane\{\plane$ becomes open$\}$,
so we have
\[
\prob(\D\setminus E) \le{d \choose2}n^{d-2} \prob \bigl(\{{\plane }
\mbox { becomes open}\}\setminus E \bigr).
\]

Now if $\plane$ becomes open but $E$ does not occur, then since each
point in $\plane$ becomes open, they must all violate one of (1), (2) or (3). By the
pigeonhole principle, at least $n^2/3$ of these points must together
violate a single condition. We will check that the probabilities of
these three cases are $o(n^{-(d-2)})$. In fact, we will see
that they are exponentially small by reducing each case to a large
deviation probability
involving a Binomial random variable with a small chance of success. We
will use
the fact that neighborhoods of two points in $\plane$ do not intersect
outside~$\plane$.
\begin{itemize}
\item$\prob(n^2/3$ vertices in $\plane$ are initially open$)$ is
exponentially small
in $n^2$, as $p=o(1)$.

\item$\prob(n^2/3$ vertices in $\plane$ are each on a line with
$A/2$ points initially open$)$
is exponentially small in $n$.

As every line covers at most $n$ points in $\plane$, this event
implies that there
are at least $n/(3d)$ parallel lines, in some direction $e_i$,
each with at least $A/2$ points initially open.
The probability that a given line has at least $A/2$ points initially
open is $O((np)^{\lfloor A/2\rfloor})=o(1)$, thus the probability that
$n/(3d)$ lines in a given direction
$e_i$ satisfy this is exponentially small in $n$.

\comment{
The probability of this is at most
\[
2{n \choose n/6} \bigl({n \choose A/2}p^{A/2} \bigr)^{n/6}
\leq2\cdot2^n \bigl(n^{-\beta A/2} \bigr)^{n/6}\leq2
\cdot2^n n^{-n/6} =o \bigl(n^{-(d-2)} \bigr).
\]
}

\item$\prob(n^2/3$ vertices in $\plane$ each have at least $A$
initially open vertices
in their neighborhoods$)$ is exponentially small in $n$.

If a vertex $w$ has at least $A$ initially open vertices in its neighborhood
then either one of the two lines through $w$ in $\plane$ contain at
least $A/4$
initially open vertices or the $d-2$ lines through $w$ not in $\plane$ together
contain at least $A/2$ initially open vertices.
This implies that either (a) there are at least $n/12$ parallel lines
in $\plane$ with at least $A/4$ vertices initially open, or (b) there
are at least $n^2/6$ vertices with at
least $A/2$ vertices in their neighborhoods outside of
$\plane$.

The probability of (a) is exponentially small by the same argument as
in the previous case.
For a fixed $w$, the probability that $(d-2)(n-1)$ sites in $\mathcal
{N}(w)\setminus\plane$
contain at least $A/2$ initially open sites
is again $O(np)=o(1)$. Thus, the probability of (b) is exponentially
small in $n^2$.
\comment{
The probability of the first is at most
\[
2{n \choose n/12} \bigl({n \choose A/4}p^{A/4} \bigr)^{n/12}
\leq2\cdot2^n \bigl(n^{-\beta A/4} \bigr)^{n/12}\leq
n^{-cn}=o \bigl(n^{-(d-2)} \bigr)
\]
and the second has probability at most
\[
{n^2 \choose n^2/6} \bigl({dn \choose
A/2}p^{A/2} \bigr)^{n^2/6}\leq c \cdot2^{n^2}
\bigl(n^{-\beta A/2} \bigr)^{n^2/6}\leq n^{-cn^2}=o
\bigl(n^{-(d-2)} \bigr).
\]
}
\end{itemize}
Therefore, $\prob(\D\setminus E)$ goes to 0 exponentially fast.
\end{pf}

\begin{pf*}{Proof of Theorem~\ref{epl}} To get the lower bound
set, $A=\lfloor\threshold-\sqrt{\threshold}\rfloor$. Then Lemmas
\ref{cajunsprel1}--\ref{cajuns}
are (for large enough $\theta$) satisfied with
\[
\beta= \frac{2}{\threshold} + \frac{4d^2 + 3}{\threshold^{3/2}}.
\]
The upper bound was proved in Theorem~\ref{2d-largetheta-upper-thm}.
\end{pf*}

\section{Further questions and conjectures}
\label{open}

We begin with a general form of threshold probabilities; we believe
that the answer to the question
below is positive.

%
\begin{question} Do there exist positive constants $c_1=c_1(i,d)$ and
$c_{3/2}=c_{3/2}(i,d)$, so that, for
all $i$ and $d$, a lower bound and an upper bound for $p_{\mathrm{c}}(i,d)$ are
both of the form
\[
n^{-1 - \sfrac{c_1}{\threshold} - \sfrac{c_{3/2}}{\threshold
^{3/2}}+o(\threshold^{-3/2})}
\]
for large enough $n$?
\end{question}

We next ask whether it is possible that generation of open planes
does not likely lead to spanning of the entire graph when $d\geq4$.

%
%
\begin{question}
Can one find $d$ and $\theta>2$ such that $\log_n(p_{\mathrm{c}}(2,d))-\log
_n(p_{\mathrm{c}}(d,d))$ is bounded away from 0 as $n \to\infty$, that is,
$p_{\mathrm{c}}(2,d)\approx n^{-\zeta}$ and $p_{\mathrm{c}}(d,d)\approx n^{-\xi}$
with $\zeta>\xi$? Does this hold for all $\theta$ and $d \geq4$?
Note that it does not hold for $d=3$ by (\ref{cloudforestcafe}).
\end{question}

It would be desirable to have a general method to determine the
critical exponent for any given
(small) $d$ and $\threshold$;
here we merely recall the simplest unsolved instances.

%
%
\begin{question}
When $d=3$, we know the critical exponents for $\theta= 2,3,4,5,6,7,9,11$;
what are the correct exponents for $\theta= 8, 10$ and $\theta\geq12$?
\end{question}

%
\begin{appendix}
\setcounter{theorem}{0}

\section*{Appendix: Poisson convergence for \texorpdfstring{$\lowercase{d}=\theta=3$}{$\lowercase{d}=theta=3$}}
\label{ap:poisson}
In this section, we outline the proof of Lemma~\ref{3d-poisson-lem}
regarding Poisson convergence of the random variables that count the
configurations that lead to spanning when $d=\threshold=3$ and $p =
an^{-2}$. Our approach is to apply the Chen--Stein method \cite
{poissonbook}, and to do so we need to introduce some notation.

We want to show that a collection of random variables, which are sums
of indicator random variables, converge to independent Poisson random
variables in the limit. That is, suppose we have disjoint sets of
indices, $\Gamma_1, \Gamma_2, \ldots, \Gamma_\ell$, let $\Gamma=
\bigcup_{i=1}^\ell\Gamma_i$, and for each $\gamma\in\Gamma$ suppose
$I_\gamma$ is an indicator random variable. For $i=1, \ldots, \ell$
let $W_i = \sum_{\gamma\in\Gamma_i} I_\gamma$ and suppose that
$EW_i = \lambda_i$ and $EI_\gamma= p_\gamma$. In our application,
the index sets are going to be $V$ for the indicators of the basic and
enhanced basic events, and $\lineset$ for the indicators of the line,
$\emptyset$-line, enhanced line and nonenhanced line events.

To apply the Chen--Stein method in many cases, we need to construct a
coupling for every fixed $\gamma\in\Gamma$ between $I_\eta$ and
$J_{\eta\gamma}$ so that
%
\begin{equation}
\label{coupling} (J_{\eta\gamma})_{\eta\neq\gamma} \deq(I_\eta|
I_\gamma=1)_{\eta\neq\gamma}.
\end{equation}
Many of the indicators that we have constructed are increasing
functions of $\omega_0$, which makes those sets of indicators
positively related (\cite{poissonbook}, Section~2.1). However, the
$\emptyset$-line and nonenhanced line indicators, $I^{\emptyset
\mathrm{L}}_\Ln$ and
$I^{\mathrm{NEL}}_\Ln$, are not increasing functions of $\omega_0$,
so whenever
these appear we are unable to use the simpler form of the Poisson
convergence theorem. Instead, we will explicitly define the couplings
below, and use Theorem~10.J of \cite{poissonbook}, which we state
below as Lemma~\ref{chen-stein}.

Suppose $X$ and $Y$ are two $\mathbb{Z}^m$-valued random variables
with laws $\mu_X$ and $\mu_Y$, and recall that the total variation
distance between $\mu_X$ and $\mu_Y$ (or with an abuse of notation,
between $X$ and $Y$ or $X$ and $\mu_Y$) is
\[
d_{\mathrm{TV}}(X,Y) = d_{\mathrm{TV}}(\mu_X,
\mu_Y):= \sup_{A\subseteq\mathbb{Z}^m} \bigl\llvert
\mu_X(A) - \mu_Y(A) \bigr\rrvert= \frac{1}{2}
\sum_{k \in
\mathbb{Z}^m} \bigl\llvert\mu_X(k) -
\mu_Y(k) \bigr\rrvert.
\]
Let $P_\lambda$ denote the law of a $\operatorname{Poisson}(\lambda
)$ random
variable (taking values in $\mathbb{Z}_{+}$). The Chen--Stein method
gives us the following bound on the total variation distance between
the joint law of $(W_1, W_2, \ldots, W_m)$ and $\prod_{i=1}^m
P_{\lambda_i}$.

%
\begin{lemma}[(\cite{poissonbook}, Theorem~10.J and Corollary~10.J.1)]
\label{chen-stein}
If $W_i$ are defined as above with $\lambda_i = EW_i$ for $i=1, \ldots,\ell$, with $EI_\gamma= p_\gamma$, then
%
\begin{equation}
\label{cs1} d_{\mathrm{TV}} \Biggl((W_1, \ldots,
W_m), \prod_{i=1}^m
P_{\lambda_i} \Biggr) \leq\sum_{\gamma\in\Gamma}
p_\gamma^2 + \mathop{\sum_{\gamma, \eta\in\Gamma}}_{\gamma\neq\eta}
p_\gamma\E\llvert J_{\eta
\gamma} - I_\eta\rrvert.
\end{equation}
If $\{I_\gamma\}_{\gamma\in\Gamma}$ are positively related then
%
\begin{equation}
\label{cs2} d_{\mathrm{TV}} \Biggl((W_1, \ldots,
W_m), \prod_{i=1}^m
P_{\lambda_i} \Biggr) \leq\sum_{\gamma\in\Gamma}
p_\gamma^2 + \mathop{\sum_{\gamma, \eta\in\Gamma}}_{\gamma\neq\eta}
\Cov(I_\gamma, I_\eta).
\end{equation}
\end{lemma}

%
\begin{remark}
In all of our applications of Lemma~\ref{chen-stein}, the first sum on
the right-hand side is easy to control, since it merely requires that
$p_\gamma$ are uniformly small. In the case of events indexed by
$\lineset$ this sum is $O(n^{-2})$, since there are $O(n^2)$ summands
and the probability of a line configuration is $O(n^2p^2) = O(n^{-2})$.
Similarly, in the case of basic or enhanced basic events this sum is
$O(n^{-3})$. The important part of the right-hand side is the term $\E
\llvert J_{\eta\gamma} - I_\eta\rrvert =
\prob{(J_{\eta\gamma}
\neq
I_\eta)}$, which requires bounding the probability that our coupling
destroys or creates the event indicated by $I_\eta$. In the case of
positively related indicators, no explicit coupling is needed, and we
must merely bound the covariances between the relevant indicators.
\end{remark}

\subsection*{Construction of couplings}
Observe that in equation (\ref{goodevent}), the last term involves
random variables that are sums of indicators that are not positively
related. So, for each of the indicators $I^{\mathrm{B}}_{\vect{v}},
I^{\emptyset
\mathrm{L}}_\Ln,
I^{\mathrm{EL}}_{\Ln}, I^{\mathrm{NEL}}_{\Ln}$ and every $\vect
{v}\in V$ and $\Ln\in
\lineset$, we must construct a suitable coupling between all of the
remaining indicators and their conditioned versions as in (\ref
{coupling}). As in (\ref{coupling}), we will use the letter $J$ for
coupled indicator random variables.

Once we show that these random variables appearing in the last term of
(\ref{goodevent}) converge jointly to independent Poissons, we will be
able to compute the limiting probabilities for all of the terms except
the second, which involves the $\ebasic$ and $\lin$ random variables.
We will treat this term separately using the simpler form of Lemma~\ref
{chen-stein}, since the enhanced basic and line indicators are
positively related.

Our goal is to show that the second summation in (\ref{cs1}) is
$O(n^{-1})$ under the couplings that we construct. We will need to
construct four couplings, one for each type of indicator, and for each
coupling we have four comparisons (to each of the four types of
indicators) that need to be made. Furthermore, for each comparison,
there are several cases that need to be checked depending on the
relative positions of the vertices and lines that index each event.
There are many cases that need to be verified, but the arguments
quickly become repetitive, thus we merely outline the proof and give
complete details in two typical cases (see proofs of Lemmas \ref
{nel-coupling-lem} and \ref{ebasic-line-lem}).

We begin with the simplest case, the \emph{basic coupling} for
conditioning on $I^{\mathrm{B}}_{\vect{v}} = 1$ for a fixed $\vect
{v}\in V$. In
this case, we merely need each of the three lines containing $\vect
{v}$ to contain at least one open vertex. To achieve this, we extend
the probability space by possibly resampling the vertices in each of
the three lines until this condition is met. That is, if a line through
$\vect{v}$ already contains an open vertex, nothing is resampled for
that line, and the original configuration is kept, otherwise it is
repeatedly replaced with an independent configuration until it does
contain an open vertex. Also, it is important to note that none of the
other vertices in the initial configuration, $\omega_0$, are altered.
Then $J^{\mathrm{B}}_{\vect{wv}}, J^{\emptyset\mathrm{L}}_{\Ln
\vect{v}}, J^{\mathrm{EL}}_{\Ln\vect{v}},
J^{\mathrm{NEL}}_{\Ln\vect{v}}$ are the indicator random variables
of the
corresponding events after the local resampling is completed. Since
$\vect{v}$ is fixed and the Hamming torus is transitive, we will drop
the index $\vect{v}$ in the conditioning on $I^{\mathrm{B}}_{\vect
{v}}=1$. %

%
\begin{lemma}
\label{basic-coupling-lem}
Under the basic coupling, the following sums are all $O(n^{-1})$:
\begin{eqnarray*}
\begin{array} {r@{\qquad}l} \displaystyle \sum_{\vect{v}\in V}
\displaystyle \mathop{\sum_{\vect{w}\in V}}_{ \vect
{w}\neq\vect{v}}
EI^{\mathrm{B}}_{\vect{v}}
\prob{ \bigl( I^{\mathrm{B}}_{\vect{w}}
\neq J^{\mathrm{B}}_{\vect{w}} \bigr)}, & \displaystyle \sum
_{\vect{v}\in V}\displaystyle \sum_{\Ln\in
\lineset} E
I^{\mathrm{B}}_{\vect{v}}
\prob{ \bigl(I^{\emptyset\mathrm{L}}_\Ln
\neq J^{\emptyset\mathrm{L}}_\Ln \bigr)},
\\\noalign{\vspace*{10pt}}
\displaystyle \sum_{\vect{v}\in V} \displaystyle \sum
_{\Ln\in\lineset} EI^{\mathrm{B}}_{\vect{v}}
\prob{ \bigl(I^{\mathrm{NEL}}_\Ln \neq J^{\mathrm{NEL}}_\Ln
\bigr)}, & \displaystyle \sum_{\vect{v}\in V} \displaystyle
\sum_{\Ln\in\lineset} EI^{\mathrm{B}}_{\vect{v}}
\prob{ \bigl(I^{\mathrm
{EL}}_\Ln \neq J^{\mathrm{EL}}_\Ln
\bigr)}. \end{array}
\end{eqnarray*}
\end{lemma}

The next simplest coupling is the \emph{$\emptyset$-line coupling}
for the conditioning on $I^{\emptyset\mathrm{L}}_\Ln= 1$ for a fixed
$\Ln\in\lineset$.
For this coupling, we need the line $\Ln$ to contain at least two
initially open vertices, so we first resample the vertices in $\Ln$ if
necessary until this condition is met. Given the locations of the open
vertices in $\Ln$, we need the two planes containing $\Ln$ to have no
open vertices that are not neighbors of the open vertices in $\Ln$. To
achieve this, we simply remove any violating vertices from $\omega_0$.
In the next three lemmas, we
use indicators $J$, with proper subscripts and superscripts, in an
analogous fashion as in Lemma~\ref{basic-coupling-lem}.

%
%
\begin{lemma}
\label{ol-coupling-lem}
Under the $\emptyset$-line coupling, the following sums are $O(n^{-1})$
\begin{eqnarray*}
\begin{array} {r@{\qquad}l} \displaystyle \sum_{\Ln\in\lineset}
\displaystyle \sum_{\vect{w}\in V} EI^{\emptyset\mathrm
{L}}_\Ln

\prob{ \bigl(I^{\mathrm{B}}_{\vect{w}} \neq J^{\mathrm{B}}_{\vect{w}}
\bigr)}, & \displaystyle \sum_{\Ln\in\lineset}\displaystyle
\mathop{ \sum_{\Ln' \in\lineset}}_{ \Ln
'\neq\Ln}
EI^{\emptyset\mathrm{L}}_\Ln
\prob{ \bigl(I^{\emptyset
\mathrm{L}}_{\Ln'}
\neq J^{\emptyset\mathrm{L}}_{\Ln
'} \bigr)},
\\\noalign{\vspace*{10pt}}
\displaystyle \sum_{\Ln\in\lineset} \displaystyle \sum
_{\Ln' \in\lineset} EI^{\emptyset
\mathrm{L}}_\Ln
\prob{
\bigl(I^{\mathrm{NEL}}_{\Ln'} \neq J^{\mathrm{NEL}}_{\Ln'}
\bigr)}, & \displaystyle \sum_{\Ln\in\lineset} \sum
_{\Ln' \in\lineset} EI^{\emptyset
\mathrm{L}}_\Ln
\prob{
\bigl(I^{\mathrm{EL}}_{\Ln'} \neq J^{\mathrm{EL}}_{\Ln'}
\bigr)}. \end{array}
\end{eqnarray*}
\end{lemma}

Next, we construct the \emph{enhanced line coupling} for the
conditioning on $I^{\mathrm{EL}}_\Ln= 1$ for a fixed $\Ln\in
\lineset$. To
achieve this, we will need the line $\Ln$ to contain at least two open
vertices, so we first resample the vertices in $\Ln$ if necessary
until this condition is met. Next, given the locations of the open
vertices in $\Ln$, we need that at least one of the two planes
containing $\Ln$ has at least one open vertex that is not collinear
with an open vertex in $\Ln$. Again, if necessary, we resample these
two planes (excepting the vertices in $\Ln$) simultaneously until this
condition is satisfied. At this point, if one of the two planes
containing $\Ln$ has at least two nonneighboring open vertices, then
the coupling is completed. Otherwise, conditional on the location of
the open vertex (or vertices) in $\mathcal{N}(\Ln)$, we need there to
be at least one open vertex in the same plane as this vertex (or
vertices) but not in the same line. If one does not exist, then we
resample the two (or four) planes containing the open vertex (or
vertices) in $\mathcal{N}(\Ln)$ but not containing $\Ln$ until there
is at least one open vertex in any of these planes [we do not resample
the vertices in $\Ln$, $\mathcal{N}(\Ln)$, or the neighborhood of
the open vertices in $\mathcal{N}(\Ln)$].

%
\begin{lemma}
\label{el-coupling-lem}
Under the enhanced line coupling, the following sums are $O(n^{-1})$:
\begin{eqnarray*}
\begin{array} {r@{\qquad}l} \displaystyle \sum_{\Ln\in\lineset}
\displaystyle \sum_{\vect{w}\in V} EI^{\mathrm{EL}}_\Ln

\prob{ \bigl(I^{\mathrm{B}}_{\vect{w}} \neq J^{\mathrm{B}}_{\vect{w}}
\bigr)}, & \displaystyle \sum_{\Ln\in\lineset} \displaystyle
\sum_{\Ln' \in\lineset} EI^{\mathrm
{EL}}_\Ln
\prob{
\bigl(I^{\emptyset\mathrm{L}}_{\Ln'} \neq J^{\emptyset\mathrm
{L}}_{\Ln'}
\bigr)},
\\\noalign{\vspace*{10pt}}
\displaystyle \sum_{\Ln\in\lineset} \displaystyle \sum
_{\Ln' \in\lineset} EI^{\mathrm
{EL}}_\Ln
\prob{
\bigl(I^{\mathrm{NEL}}_{\Ln'} \neq J^{\mathrm{NEL}}_{\Ln'}
\bigr)}, & \displaystyle \sum_{\Ln\in\lineset} \displaystyle
\mathop{\sum_{\Ln' \in\lineset}}_{ \Ln
'\neq\Ln}
EI^{\mathrm{EL}}_\Ln
\prob{ \bigl(I^{\mathrm{EL}}_{\Ln
'}
\neq J^{\mathrm{EL}}_{\Ln
'} \bigr)}. \end{array}
\end{eqnarray*}
\end{lemma}

Finally, we construct the \emph{nonenhanced line coupling} for the
conditioning on $I^{\mathrm{NEL}}_\Ln= 1$ for a fixed $\Ln\in
\lineset$. To
achieve this, we will need the line $\Ln$ to contain at least two open
vertices. So, first we resample the vertices in $\Ln$ if necessary
until this condition is met. Next, given the locations of the open
vertices in $\Ln$, we need: (1) that at least one of the two planes
containing $\Ln$ has at least one open vertex that is not collinear
with an open vertex in $\Ln$, and (2) that neither plane containing
$\Ln$ has more than one noncollinear open vertex. Again, if necessary,
we resample these two planes simultaneously until these conditions are
met (here we do not resample $\Ln$). Now, conditional on the locations
of the open points in $\mathcal{N}(\Ln)$, we must guarantee that
there are no other points outside of $\Ln$ that are coplanar but not
collinear with these points. For this part of the coupling, we simply
remove any violating points from $\omega_0$.

%
\begin{lemma}
\label{nel-coupling-lem}
Under the nonenhanced line coupling, the following sums are $O(n^{-1})$:
\begin{eqnarray*}
\begin{array} {r@{\qquad}l} \displaystyle \sum_{\Ln\in\lineset}
\displaystyle \sum_{\vect{w}\in V} EI^{\mathrm{NEL}}_\Ln

\prob{ \bigl(I^{\mathrm{B}}_{\vect{w}} \neq J^{\mathrm{B}}_{\vect{w}}
\bigr)}, & \displaystyle \sum_{\Ln\in\lineset} \displaystyle
\sum_{\Ln' \in\lineset} EI^{\mathrm
{NEL}}_\Ln
\prob{
\bigl(I^{\emptyset\mathrm{L}}_{\Ln'} \neq J^{\emptyset\mathrm
{L}}_{\Ln'}
\bigr)},
\\\noalign{\vspace*{10pt}}
\displaystyle \sum_{\Ln\in\lineset} \displaystyle \mathop{\sum
_{\Ln' \in\lineset}}_{ \Ln
'\neq\Ln} EI^{\mathrm{NEL}}_\Ln
\prob{ \bigl(I^{\mathrm
{NEL}}_{\Ln'}\neq
J^{\mathrm{NEL}}_{\Ln
'} \bigr)}, & \displaystyle \sum
_{\Ln\in\lineset} \displaystyle \sum_{\Ln' \in\lineset}
EI^{\mathrm
{NEL}}_\Ln
\prob{ \bigl(I^{\mathrm{EL}}_{\Ln'}
\neq J^{\mathrm{EL}}_{\Ln'} \bigr)}. \end{array}
\end{eqnarray*}
\end{lemma}

\begin{pf}
We now outline the proof by bounding the first summation above. There
are three cases.

\emph{Case 1}: $\vect{w}\in\Ln$. This term appears in the sum
$O(n^3)$ times, and $EI^{\mathrm{NEL}}_\Ln= O(n^{-2})$, so we must
show that
$
\prob{(I^{\mathrm{B}}_{\vect{w}} \neq J^{\mathrm{B}}_{\vect{w}})} =
O(n^{-2})$. Now
there are two subcases, \emph{destruction} and \emph{creation,}
respectively: $
\prob{(I^{\mathrm{B}}_{\vect{w}} = 1, J^{\mathrm{B}}_{\vect{w}}
= 0)}$ and
$
\prob(I^{\mathrm{B}}_{\vect{w}} = 0,\allowbreak    J^{\mathrm{B}}_{\vect{w}} = 1)$.
Clearly, $
\prob
{(I^{\mathrm{B}}_{\vect{w}} = 1, J^{\mathrm{B}}_{\vect{w}} =
0)}\leq
\prob{(I^{\mathrm{B}}
_{\vect
{w}} = 1)} = O(n^{-3})$. Next, in order for the creation event to
occur, the resampling procedure must have generated at least one open
vertex in both planes containing $\Ln$, and both of these points must
lie in the neighborhood of $\vect{w}$. The probability of this is
$O(n^{-2})$, since we require an open vertex in each of two fixed lines.

\emph{Case 2}: $\vect{w} \in\mathcal{N}(\Ln)$. This term appears
in the sum $O(n^4)$ times, and $EI^{\mathrm{NEL}}_\Ln= O(n^{-2})$, so
we must
show that $
\prob{(I^{\mathrm{B}}_{\vect{w}} \neq J^{\mathrm{B}}_{\vect{w}})} =
O(n^{-3})$. Once again, there are two subcases as above. The creation
event cannot occur in this case because an open vertex in $\mathcal
{N}(\Ln)$ that is collinear with $\vect{w}$ must not see any coplanar
open vertices (off of $\Ln$), which includes a line in the
neighborhood of $\vect{w}$, so $\vect{w}$ can no longer see an open
vertex in each direction. The probability of the destruction event can
be trivially bounded by $O(n^{-3})$ as in Case 1.

\emph{Case 3}: $\vect{w} \notin\mathcal{N}(\Ln) \cup\Ln$. This
term appears
in the sum $O(n^5)$ times, and $EI^{\mathrm{NEL}}_\Ln= O(n^{-2})$, so
we must
show that
$
\prob{(I^{\mathrm{B}}_{\vect{w}} \neq J^{\mathrm{B}}_{\vect{w}})} =
O(n^{-4})$. Once
again, the creation
event cannot occur for the same reason as cited in Case 2. The
destruction event can
only occur if one of the initially open points in the neighborhood of
$\vect{w}$ is in
one of the resampled planes. At most six planes are affected with probability
$1-O(n^{-1})$, and with the same probability none of the resampled
planes contain a line
in the neighborhood of $\vect{w}$. The probability of the destruction
event is at most
$O(n^{-4})$, since $\vect{w}$ must first have three open neighbors
initially [an event
with probability $O(n^{-3})$], and at least one must coincide with one
of the resampled
planes [an event with probability $O(n^{-1})$].
\end{pf}

\subsection*{Positively related case}
Since $\{I^{\mathrm{EB}}_{\vect{v}}\}_{\vect{v}\in V}$ and $\{
I^{\mathrm{L}}_{\Ln}\}
_{\Ln\in\lineset}$ are all increasing functions of $\omega_0$,
these collections of indicators are positively related so we may apply
the simpler form of Lemma~\ref{chen-stein} by bounding the covariances.

%
\begin{lemma}
\label{ebasic-line-lem}
The collections of indicators $\{I^{\mathrm{EB}}_{\vect{v}}\}_{\vect
{v}\in V}$
and $\{I^{\mathrm{L}}_{\Ln}\}_{\Ln\in\lineset}$ are positively
related and
the following sums are $O(n^{-1})$:
\[
\sum_{\vect{v}\in V} \mathop{\sum
_{\vect{w}\in V}}_{ \vect{w}\neq
\vect{v}} \Cov \bigl(I^{\mathrm{EB}}_{\vect{v}},I^{\mathrm
{EB}}_{\vect{w}}
\bigr), \qquad \sum_{\vect{v}\in V} \sum
_{\Ln\in\lineset} \Cov \bigl(I^{\mathrm
{EB}}_{\vect
{v}},I^{\mathrm{L}}_{\Ln}
\bigr), \qquad \sum_{\Ln\in\lineset} \mathop{\sum
_{\Ln'\in\lineset}}_{ \Ln
'\neq\Ln} \Cov \bigl(I^{\mathrm{L}}_{\Ln},I^{\mathrm{L}}_{\Ln'}
\bigr).
\]
\end{lemma}

Note that the bound on the last sum, which involves only indicators of
line events, is implied by combining the results for the enhanced line
and nonenhanced line couplings in Lemmas \ref{el-coupling-lem} and
\ref{nel-coupling-lem} by writing $I^{\mathrm{L}}_\Ln= I^{\mathrm
{EL}}_\Ln+ I^{\mathrm{NEL}}_\Ln$.

\begin{pf}
We will explain the proof of the bound on the first sum, as the second
sum is evaluated in a similar fashion and the third is implied by
previous lemmas. We break up the sum into three cases depending on the
Hamming distance between $\vect{v}$ and $\vect{w}$.

\emph{Case 1}: $d(\vect{v},\vect{w})=1$. There are $O(n^4)$ such
terms in the sum, so we need to show that the covariance is
$O(n^{-5})$. In this case it suffices to use the trivial bound $\Cov
(I^{\mathrm{EB}}_{\vect{v}},I^{\mathrm{EB}}_{\vect{w}}) \leq\E
I^{\mathrm{EB}}_{\vect{v}}I^{\mathrm{EB}}_{\vect
{w}} = \prob(\event_v^{\mathrm{EB}}\cap\event_w^{\mathrm{EB}})$, which is the
probability that an enhanced basic configuration appears at $\vect{v}$
and at $\vect{w}$. For this event to occur, $\vect{v}$ must have one
open neighbor in each direction, one of which is shared with $\vect
{w}$, so $\vect{w}$ needs only one open neighbor in each direction
orthogonal to $\vect{w}-\vect{v}$. This is a total of at least five
open points on five fixed lines, which has probability $O((np)^5) =
O(n^{-5})$ as desired.

\emph{Case 2}: $d(\vect{v},\vect{w}) =2$. There are $O(n^{5})$ such
terms in the sum, so we need to show that the covariance is
$O(n^{-6})$. Again, it suffices to use the bound $\Cov(I^{\mathrm
{EB}}_{\vect
{v}},I^{\mathrm{EB}}_{\vect{w}}) \leq\E I^{\mathrm{EB}}_{\vect
{v}}I^{\mathrm{EB}}_{\vect{w}}$. In
this case, the vertices $\vect{v}$ and $\vect{w}$ have exactly two
common neighbors, so there are three cases: zero, one, or two of these
common neighbors are initially open. If neither common neighbor is
initially open, then $\vect{v}$ and $\vect{w}$ each independently
need one open neighbor in each direction---a total of six open
vertices in six fixed lines, which has probability $O(n^{-6})$. If one
of the common neighbors is open, an event with probability $O(p) =
O(n^{-2})$, then $\vect{v}$ and $\vect{w}$ each need an open neighbor
in two other directions---a total of four open vertices in four fixed
lines which has probability $O(n^{-4})$. This gives a probability of
$O(n^{-6})$ to the case where one common neighbor is open. The event
that both common neighbors are open has probability $p^2 = O(n^{-4})$,
and $\vect{v}$ and $\vect{w}$ each require one more occupied neighbor
in one direction, which has probability $O(n^{-2})$ for a total
probability of $O(n^{-6})$.

\emph{Case 3}: $d(\vect{v},\vect{w})=3$. There are $O(n^{6})$ such
terms in the sum, so we need to show that the covariance is
$O(n^{-7})$, and the trivial upper bound on the covariance will not
suffice. Observe that the planes containing $\vect{v}$ and the planes
containing $\vect{w}$ intersect only along 6 lines, and conditional on
the event that none of the points on these lines are initially open,
$I^{\mathrm{EB}}_{\vect{v}}$ and $I^{\mathrm{EB}}_{\vect{w}}$ are
independent. Call this
event $E_{\mathrm{empty}}$, then since $I^{\mathrm{EB}}_{\vect{v}}$ and
$I^{\mathrm{EB}}_{\vect
{w}}$ are increasing functions of $\omega_0$, the covariance is
bounded by
\[
\Cov \bigl(I^{\mathrm{EB}}_{\vect{v}},I^{\mathrm{EB}}_{\vect{w}}
\bigr) \leq
\prob{ \bigl(I^{\mathrm{EB}}_{\vect
{v}}I^{\mathrm{EB}}_{\vect{w}}
= 1, E_{\mathrm{empty}}^c \bigr)}.
\]
We now divide the event $E_{\mathrm{empty}}^c$ into subcases according to
which vertices in the intersection are open. There are two types of
vertices in the intersection---those which are neighbors to either
$\vect{v}$ or $\vect{w}$, and those which are only in the same plane
as each vertex. There are exactly 6 vertices in the former category and
$6(n-2)$ in the latter. The probability that $j$ of the 6 vertices in
$[\mathcal{N}(\vect{v})\cap\mathcal{N}(\mathcal{N}(\vect{w}))]
\cup[\mathcal{N}(\vect{w})\cap\mathcal{N}(\mathcal{N}(\vect
{v}))]$ are initially open is $O(p^j) = O(n^{-2j})$. Conditional on
this, $\vect{v}$ and $\vect{w}$ collectively require an initially
open vertex in each of the remaining $6-j$ lines in their
neighborhoods, which has probability $O(n^{-6+j})$, giving a total
probability of $O(n^{-6-j})$ to the event that there are $j$ of these 6
vertices initially open and both enhanced basic events occur.
Therefore, if $j\geq1$ we are done, otherwise we must consider the
case where $j=0$ and then $E_{\mathrm{empty}}^c$
requires that at least one\vspace*{1pt} vertex among the $6(n-2)$ vertices in
$\mathcal{N}(\mathcal{N}(\vect{v}))\cap\mathcal{N}(\mathcal
{N}(\vect{w}))$ are initially open. This event has probability $O(np)
= O(n^{-1})$, and when $j=0$, $\vect{v}$ and $\vect{w}$ still need
one open vertex in each line of their neighborhoods, which has
probability $O(n^{-6})$, giving a total probability of $O(n^{-7})$.
\end{pf}

\begin{pf*}{Proof of Lemma~\ref{3d-poisson-lem}}
The limiting means are straightforward to calculate, as outlined in
Remark \ref{3d-poisson-rem}. It is also not difficult to show that
$d_{\mathrm{TV}}(P_{\lambda_n},P_{\lambda}) \leq\llvert\lambda_n-\lambda
\rrvert$
so if $\lambda_n \to\lambda$ then $P_{\lambda_n}$ converges to
$P_\lambda$. Therefore, applying Lemma~\ref{chen-stein} and using
Lemmas \ref{basic-coupling-lem}--\ref{nel-coupling-lem} to bound the
second summation in (\ref{cs1}) implies that the random variables
$\basic$, $\oline_i$, $\neline_i$ and $\eline$ (where $i=1,2,3$, so
there are a total of 8 random variables) converge jointly to
independent Poisson random variables with the appropriate limiting
means. \mbox{Similarly}, applying Lemma~\ref{chen-stein} and using Lemma~\ref
{ebasic-line-lem} to bound the second summation in (\ref{cs2}) implies
that the random variables $\ebasic$ and $\lin$ converge jointly to
independent Poisson random variables with the appropriate limiting means.
\end{pf*}
\end{appendix}



%

\printaddresses

\end{document}